\documentclass{article}%{article}

\usepackage{arxiv}

\usepackage[utf8]{inputenc} % allow utf-8 input
\usepackage[T1]{fontenc}    % use 8-bit T1 fonts
\usepackage{hyperref}       % hyperlinks
\usepackage{url}            % simple URL typesetting
\usepackage{booktabs}       % professional-quality tables
\usepackage{amsfonts}       % blackboard math symbols
\usepackage{nicefrac}       % compact symbols for 1/2, etc.
\usepackage{microtype}      % microtypography
\usepackage{lipsum}
\usepackage{graphicx}
\usepackage{subcaption}
\usepackage[numbers]{natbib}
\usepackage{doi}
\usepackage{enumitem} %Customize itemize
\usepackage{dsfont}
\usepackage{amsmath}
\usepackage{geometry}
\usepackage[english]{babel}
\usepackage{amsthm}
\usepackage[dvipsnames]{xcolor}
\usepackage{comment}
\usepackage{algorithm}
\usepackage{algpseudocode}
\usepackage{multirow}

\graphicspath{ {./images/} }

\newtheorem{theorem}{Theorem}
\newtheorem{corollary}{Corollary}[theorem]
\newtheorem{proposition}{Proposition}
\newtheorem{remark}{Remark}

\title{A note on Refracted Skew Brownian Motion with an application}

\newif\ifuniqueAffiliation
\uniqueAffiliationtrue

\ifuniqueAffiliation % Standard variant of author block
\author{ \href{https://orcid.org/0000-0003-0145-1252}{\includegraphics[scale=0.06]{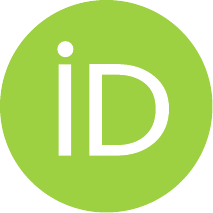}\hspace{1mm}Zaniar Ahmadi} \thanks{
We appreciate Natural Sciences and Engineering Research Council of Canada (NSERC) for its financial support under grant number RGPIN-2021-04100.
} \\
  Mathematics and Statistics Department\\
  Concordia University\\
  Montreal, Canada \\
  \texttt{zaniar.ahmadi@concordia.ca} \\
   \And
\href{https://orcid.org/0000-0002-8205-0217}{\includegraphics[scale=0.06]{orcid.pdf}\hspace{1mm} Xiaowen Zhou }
\footnotemark[1] \thanks{Corresponding author} \\
  Mathematics and Statistics Department\\
  Concordia University\\
  Montreal, Canada \\
  \texttt{xiaowen.zhou@concordia.ca } \\
}

\begin{document}
\maketitle
\begin{abstract}
For refracted skew Brownian motion (skew Brownian motion with two-valued drift),  adopting a perturbation approach we find expressions of its potential densities. As applications, we recover its transition density and study its long-time asymptotic behaviors. In addition, we also compare with previous results on transition densities for skew Brownian motions. We propose two approaches for generating quasi-random samples by approximating the cumulative distribution function and discuss their risk measurement application.
\end{abstract}

% keywords can be removed

%\maketitle

\section{Introduction}

Skew Brownian motion (SBM for short) was first introduced by It\^{o} and McKean in 1965 in Section two of their book \cite{ito2012diffusion}.  In 1978 Walsh \cite{walsh1978diffusion} proposed a path construction of skew Brownian motion and found an expression for the transition density of skew Brownian motion. Harrison and Shepp \cite{harrison1981skew} showed that skew Brownian motion is the weak limit of a symmetric random walk on the integers with modified behavior at the origin, and also proved that it is the unique strong solution to a stochastic differential equation driven by Brownian motion together with a singular drift of its local time at $0$. In 1984 Le Gall  \cite{LeGall19893} further established the existence and pathwise uniqueness of strong solutions to a more general stochastic differential equation with a general diffusion coefficient involving local times at different levels.
%governing a skew Brownian motion with drift.  
Appuhamillage et al. in \cite{appuhamillage2011occupation} for the first time analyzed the skew Brownian motion with one drift and provided the joint distribution of key functionals of skew Brownian motion. We refer to  Lejay in \cite{lejay2006constructions} for a review of earlier work related to skew Brownian motion.

%{\color{blue} Who first studied skew Brownian motion with drift or two-valued drift? (Gairat \cite{gairat2017density} for the first time introduced it)

%Check 
%One-dimensional stochastic differential equations
%involving a singular increasing process by Barlow and Perkins.- where can we use this?}

Skew Brownian motion is a versatile concept with applications spanning diverse domains. Notably, in mathematical finance, it serves as a fundamental tool for option pricing, as evidenced by studies such as those by Corns and Satchell. \cite{corns2007skew}, Wang et al. \cite{wang2015skew}, Gairat and Shcherbakov \cite{gairat2017density}, and Zhu and He \cite{zhu2018new}. Additionally, it finds applications in the interest rate and intensity modeling, with notable contributions from Jarrow et al. \cite{jarrow1997markov}, Duffie and Singleton. \cite{duffie1999modeling}, and Decamps et al. \cite{decamps2006self}.
Calibrating skew Brownian motion parameters has been the subject of intensive research. Bardou and Martinez \cite{bardou2010statistical} introduced a statistical estimator for calibrating the coefficients of a skewed process in the case in which it is reflected over a bounded interval.  Bai et al. \cite{bai2022bayesian} advocated for a Bayesian approach to estimate the skew coefficient and skew level of the skew Ornstein-Uhlenbeck (OU) process, and Lajay and Mazzonetto \cite{lejay2023maximum} proposed the use of Maximum Likelihood Estimators to calibrate the parameters of a skew Brownian motion.
Furthermore, skew Brownian motion is instrumental in addressing dispersion in heterogeneous media. Noteworthy contributions include works by Freidlin and Sheu \cite{freidlin2000diffusion}, Ramirez \cite{ramirez2011multi}, and Appuhamillage et al. \cite{appuhamillage2011occupation}. Moreover, it proves invaluable in modeling diffusion problems with spatially discontinuous diffusion coefficients, as highlighted in studies by Zhang \cite{zhang2000calculation} and Ascione et al. \cite{ascione2024modeling}.

It is an interesting problem to identify probability laws for solutions to singular SDEs.
The trivariate density of Brownian Motion and its local and occupation time were obtained in Karatzas and Shreve \cite{karatzas1984trivariate}  and applied to problems of stochastic control.
An expression for transition density of skew Brownian motion was first found by Walsh \cite{walsh1978diffusion}.
%Appuhamillage et al. in \cite{appuhamillage2011occupation} analyzed the skew Brownian motion with one drift and provided the joint distribution of key functionals of skew Brownian motion. In 2011, Ramirez \cite{ramirez2011multi} extended the notion of skew Brownian motion to multi-skewed Brownian motion???(I think this was first done by LeGall).
%{\color{blue}Who found the density of skew Brownian motion with alternate drift? - added} 
Berezin and Zayats \cite{berezin2017skew} found an expression of transition density for skew Brownian motion with alternating drift (or skew Caughey-Dieness process as mentioned in their paper) in \cite{berezin2017skew}.  
Gairat and Schenrbakov  \cite{gairat2017density} used an approximation method for obtaining the joint distribution of both Brownian motion and skew Brownian motion with two-valued drift and their functionals. To enhance precision, the researchers approximate skew Brownian motion with two-valued drift using a discrete Markov chain. Except at skew level zero, this Markov chain evolves in the same way as a simple random walk with unit jumps. Notably, the authors employ a complex and lengthy methodology to determine transition probabilities. However, significant details about this method are excluded from their explanation. The same authors obtained the probability density of skew Brownian motion with alternating drifts in Gairat and Schenrbakov \cite{gairat2022skew}. 
In 2023, Borodin \cite{borodin2023distributions} found the potential measure of skew Brownian motion with two-valued drift using the Laplace transform method to solve the associated partial differential equation, and calculated the distributions of integral its functionals with respect to spatial variable of local time.

%Skew Brownian motion was first introduced in   McKean, Walsh, Le Gall.  Review of Lejay. add references

%Different generalizations of skew Brownian motions. add references.

%Applications. add references.

%Previous results on transition density of skew Brownian motions starting with Walsh. add references.

In this paper, we find expressions of potential measures for refracted skew Brownian motion (RSBM) and then obtain an expression for transition density by inverting the Laplace transform. Applying the above results we further discuss the asymptotic behaviors of the skew Brownian motion with two-valued drift.
To this end we adopt a new perturbation approach of Chen et al. \cite{chen2024optimal} that is different from those in \cite{borodin2023distributions} and \cite{gairat2017density}.
More precisely, we take use of the solution of RSBM in Gao et al. \cite{gao2024optimal} and the previous results on the potential measure of Brownian motion with drift. We then discuss the long-term behaviors of the skew Brownian motion such as the stationary density. We also compare our results with previous similar results on Skew Brownian motions.
  
The rest of the paper is arranged as follows. In Section \ref{sec:sbm2drifts} we introduce the skew Brownian motion with two-valued drift and its exit problems. In Section \ref{Sec_Potential} we derive expressions for the potential measures applying solutions to the exit problems together with a perturbation approach. In Section \ref{sec:Laplace_sbm} we invert the Laplace transform to find explicit formulas for the transition densities. In Section \ref{sec:limitdensity} we discuss the limiting behaviors of the refracted skew-Brownian motion. Finally, In section \ref{sec:approx_simu} we propose two approaches to estimate the cumulative distribution function to generate quasi-random samples of SBM. These approximations are also used to estimate two conventional risk metrics, Value at Risk (VaR) and Conditional Value at Risk (CVaR).

\section{Refracted Skew Brownian Motion and its exit problems}
\label{sec:sbm2drifts}
In this section, we consider a refracted skew Brownian motion, $X= (X_t)_{t \geq 0}$, defined on a filtered probability space $(\mathcal{F}, (\mathcal{F}_t)_{t \geq 0}, \mathds{P})$, which is the pathwise unique solution to SDE
\begin{equation}
\label{sde0_2drifts}
    X_t = X_0 + B_t + \beta L^X(t,a) + \mu_+ \int_{0}^{t} \mathds{1}_{\{X_s >a \} }\mathrm{d}s + \mu_- \int_{0}^{t} \mathds{1}_{\{X_s <a \} }\mathrm{d}s , \hspace{1cm} t\geq 0
\end{equation}
where $X_0, \mu_+, \mu_-\in \mathbb{R}$ and $a$ is skew level, and $\beta \in (-1,1)$, $B$ denotes a standard Brownian motion and 
\[
L^X(t,a) := \lim_{\epsilon \rightarrow 0^+} \frac{1}{2 \epsilon} \text{Leb}(s \leq t : |X_s-a| \leq \epsilon),
\]
denotes the symmetric local time process for $X$ at level $a$ up to time $t\in \mathds{R}_+$, see e.g. Le Gall \cite{LeGall19893} for existence and uniqueness of solution to (\ref{sde0_2drifts}).

\begin{remark}
    Note that for $\sigma>0$, $Y$ solves SDE
    \[    Y_t = Y_0 + \sigma B_t + \beta\sigma L^Y(t, \sigma a) + \sigma \mu_+ \int_{0}^{t} \mathds{1}_{\{Y_s >\sigma a \} }\mathrm{d}s + \sigma\mu_- \int_{0}^{t} \mathds{1}_{\{Y_s <\sigma a \} }\mathrm{d}s \]
    if and only if $X_t:=Y_t/\sigma$ solves SDE
    \[    X_t = Y_0/\sigma +  B_t + \beta L^X(t,a) + \mu_+ \int_{0}^{t} \mathds{1}_{\{X_s >a\} }\mathrm{d}s + \mu_- \int_{0}^{t} \mathds{1}_{\{X_s <a\} }\mathrm{d}s. \]
\end{remark}

For a boundary $ y \in \mathds{R}$, define the first hitting time of the boundary $y$ for the stochastic process $X_t$ to be 
\[ 
\tau_y := \inf\{t \geq 0: X_t=y \},
\]
with the convention $\inf \emptyset = \infty$. We denote the Laplace transform of $\tau_y$ as $\mathds{E}_x[e^{-q\tau_y}]$ for a fixed $q \geq 0$. For any interval $[y,z]$, where $x,y,z \in \mathds{R}$ and $y\leq x < z$, we denote the first hitting time for the lower barrier $y$ and upper barrier $z$ for the stochastic process $\{X_t\}_{\{t \geq 0\}}$ as $\tau_y$ and $\tau_z$, respectively. Then define the first exit time of the interval $(y,z)$ by
\[
\tau_{y,z} := \tau_y \wedge \tau_z = \min\{\tau_y, \tau_z\}.
\]
Moreover, the Laplace transform of $\tau_y$ for $\tau_y < \tau_z$ is expressed as $\mathds{E}_x[e^{-q\tau_y}; \tau_y < \tau_z]$. The Laplace transform of $\tau_z$ for $\tau_z < \tau_y$ is expressed as $\mathds{E}_x[e^{-q\tau_z}; \tau_z < \tau_y]$. It can be interpreted as the Laplace transform of the time of ruin before reaching the lower barrier $y$ and plays a key role in the problems of negative duration and occupation time.\\
Note that SDE (\ref{sde0_2drifts}) has a pathwise unique solution. To obtain the Laplace transform of some exit times, we first need to find the general solutions $g_{1,q}(x)$ and $g_{2,q}(x)$ in $C(\mathds{R}) \cap C^2(\mathds{R} \setminus \{a\})$ to the following equation:
\begin{equation*}
    \frac{1}{2} g''(x) + \mu_{+} \mathds{1}_{\{ x>a \}} g' (x) + \mu_{-} \mathds{1}_{\{ x<a \}} g' (x)  = q g(x), 
\end{equation*}
with $(1+\beta)g^{\prime}(a^+) = (1-\beta) g^{\prime}(a^-)$ for $\beta \in (-1,1)$ and $q \geq 0$.

Let $\delta^- = \sqrt{\mu^2_- + 2q}$, and $\rho_1^- = -(\mu_- + \delta^-)$ and $\rho_2^- = -\mu_- + \delta^-$ be the root of equation $\frac{1}{2}x^2+\mu_-x -q=0$, where $\rho_1^-<0<\rho_2^-$. Beside that, let $\delta^+ = \sqrt{\mu^2_+ + 2q}$, and $\rho_1^+ = -(\mu_+ + \delta^+)$ and $\rho_2^+ = -\mu_+ + \delta^+$ be the root of equation $\frac{1}{2}x^2+\mu_+x -q=0$, where $\rho_1^+<0<\rho_2^+$. Let
\begin{equation}
    g_{1,q}(x) = e^{\rho^+_1(x-a)} \mathbf{1}_{\{ x>a \}}
    + \Big( c_1(q) e^{\rho^-_2(x-a)} + (1-c_1(q)) e^{\rho^-_1(x-a)} \Big) \mathbf{1}_{\{ x \leq a \}}
\end{equation}
and 
\begin{equation}
    g_{2,q}(x) = \Big( (1-c_2(q)) e^{\rho^+_2(x-a)} + c_2(q) e^{\rho^+_1(x-a)} \Big)  \mathbf{1}_{\{ x>a \}} + e^{\rho^-_2(x-a)}  \mathbf{1}_{\{ x \leq a \}},
\end{equation}
where
\[
c_1(q) = \frac{(1+\beta) \rho^+_1 - (1-\beta) \rho^-_1}{(1-\beta)(\rho^-_2 - \rho^-_1)}
\]
and
\[
c_2(q) = \frac{(1+\beta)\rho^+_2 - (1-\beta) \rho^-_2}{(1+\beta)(\rho^+_2 - \rho^+_1)}.
\]
The following two theorems are presented by Gao et al. \cite{gao2024optimal} and serve as the foundation for derivations in the preceding sections. 
\begin{theorem} \label{thm1_zhou}
    For $x,y,z \in \mathds{R}$ and $y \leq x < z$, we have 
    \begin{eqnarray}
         \mathds{E}_x[e^{-q\tau_z}; \tau_z < \tau_y] = \frac{w(x,y)}{w(z,y)}
    \end{eqnarray}
    and 
    \begin{eqnarray}
         \mathds{E}_x[e^{-q\tau_y}; \tau_y < \tau_z] = \frac{w(x,z)}{w(y,z)},
    \end{eqnarray}
    where 
    \[
    w(x,y) := g_{2,q}(x)g_{1,q}(y) - g_{1,q}(x)g_{2,q}(y).
    \]
\end{theorem}
\begin{theorem} \label{thm2_zhou}
    For $x,r \in \mathds{R}$ and $x \geq r$,
    \begin{eqnarray} \label{lap1}
        \mathds{E}_x[e^{-q\tau_r}] = e^{\rho^+_1(x-r)}\mathbf{1}_{\{ r>a \}} + \frac{g_{1,q}(x)}{ c_1(q) e^{\rho^-_2(r-a)} + (1-c_1(q)) e^{\rho^-_1(r-a)}}\mathbf{1}_{\{ r\leq a \}}
    \end{eqnarray}
 and for $ x < r$, 
  \begin{eqnarray} \label{lap2}
        \mathds{E}_x[e^{-q\tau_r}] = e^{\rho^-_2(x-r)}\mathbf{1}_{\{ r\leq a \}} + \frac{g_{2,q}(x)}{(1-c_2(q)) e^{\rho^+_2(r-a)} + c_2(q) e^{\rho^+_1(r-a)}}\mathbf{1}_{\{ r> a \}}.
    \end{eqnarray}
\end{theorem}
% We can write (\ref{lap1}) as
% \begin{equation}
% \mathds{E}_x[e^{-q\tau_r}] =
%     \begin{cases}
%         e^{\rho^+_1(x-r)},  &  a \leq r \leq x ,\\
%          & \\
%         \frac{e^{\rho^+_1(x-a)}}{c_1(q) e^{\rho^-_2(r-a)} + (1-c_1(q)) e^{\rho^-_1(r-a)}}, &  r \leq a \leq x ,\\
%          & \\
%          \frac{c_1(q) e^{\rho^-_2(x-a)} + (1-c_1(q)) e^{\rho^-_1(x-a)}}{c_1(q) e^{\rho^-_2(r-a)} + (1-c_1(q)) e^{\rho^-_1(r-a)}},  & r \leq x \leq a.
%     \end{cases}
% \end{equation}
% We can also write (\ref{lap2}) as
% \begin{equation}
% \mathds{E}_x[e^{-q\tau_r}] =
%     \begin{cases}
%         e^{\rho^-_2(x-r)},  &  x < r < a,\\
%          & \\
%         \frac{e^{\rho^-_2(x-a)}}{(1-c_2(q)) e^{\rho^+_2(r-a)} + c_2(q) e^{\rho^+_1(r-a)}}, &  x < a < r, \\
%          & \\
%          \frac{(1-c_2(q)) e^{\rho^+_2(x-a)} + c_2(q) e^{\rho^+_1(x-a)}}{(1-c_2(q)) e^{\rho^+_2(r-a)} + c_2(q) e^{\rho^+_1(r-a)}},  & a < x < r.
%     \end{cases}
% \end{equation}

\section{Potential measures of RSBM}\label{Sec_Potential}
Throughout this paper, we assume $e_q$ has an exponential distribution with intensity $q$ that is independent of $X_t$.

Using the results of Theorem \ref{thm1_zhou} and Theorem \ref{thm2_zhou}, the memoryless property of exponential distribution, and the Markov property of skew Brownian motion, we can derive the following theorem. 
\begin{theorem}
    The potential measure of the refracted skew Brownian motion $X_{t}$ with skew level $a=0$ and starting point $x=0$ is given by:
    \begin{equation} \label{eq:pmeasure_x0}
\mathds{P}(X_{e_q}\in \mathrm{d}y) =
    \begin{cases}
        \frac{2(1+\beta)q}{ (1+\beta) \mu_+ - (1-\beta) \mu_- + (1+\beta)\sqrt{\mu_+^2 + 2q} + (1-\beta) \sqrt{\mu_-^2+2q}}  e^{(\mu_+ -  \sqrt{\mu^2_+ + 2q})y}\mathrm{d}y,  &  y> 0, \\
         & \\
        \frac{2(1-\beta)q}{ (1+\beta) \mu_+ - (1-\beta) \mu_- + (1+\beta)\sqrt{\mu_+^2 + 2q} + (1-\beta) \sqrt{\mu_-^2+2q}}  e^{(\mu_- +  \sqrt{\mu^2_- + 2q})y} \mathrm{d}y, &  y<0. 
    \end{cases}
\end{equation}
\end{theorem}
\begin{proof}
 Let $X_t^{1} = \mu_+ t + B_t$ be a Brownian motion and let $\tau_x^{1}$ be the first hitting time of process $X_t^{1}$. For $0<\epsilon <y$, by the Markov property and the memoryless property,
\begin{equation}
    \begin{split}
    \mathds{P}(X_{e_q} \in\mathrm{d}y) &= \mathds{P}(\tau_\epsilon<e_q, X_{e_q} \in\mathrm{d}y) \\ 
    &= \mathds{P}(\tau_\epsilon<e_q)\mathds{P}_{\epsilon}(X_{e_q} \in\mathrm{d}y) \\
    &= \mathds{E}e^{-q\tau_\epsilon} \Big( \mathds{P}_\epsilon(\tau_0<e_q, X_{e_q} \in\mathrm{d}y) + \mathds{P}_\epsilon(e_q<\tau_0, X_{e_q} \in\mathrm{d}y) \Big) \\ 
    &= \mathds{E}e^{-q\tau_\epsilon} \Big( \mathds{E}_\epsilon e^{-q\tau_0} \mathds{P}(X_{e_q} \in\mathrm{d}y) + \mathds{P}_\epsilon(e_q < \tau_0^{1}, X_{e_q}^{1} \in \mathrm{d}y) \Big).
    \end{split}
\end{equation}
Solving the above equation we have
\begin{equation} \label{limit3}
    \begin{split}
   \int_{0}^{\infty} q e^{-qt}\mathds{P}(X_t\in\mathrm{d}y) \mathrm{d}t &= \mathds{P}(X_{e_q}\in\mathrm{d}y) \\
   % &= \frac{\mathds{E}e^{-q\tau_\epsilon} \mathds{P}_\epsilon(e_q<\tau_0^{1}, X_{e_q}^{1} \in\mathrm{d}y)}{1- \mathds{E}e^{-q\tau_\epsilon} \mathds{E}_\epsilon e^{-q\tau_0}} \\
   &= \lim_{\epsilon \rightarrow 0^+} \frac{\mathds{E}e^{-q\tau_\epsilon} \mathds{P}_\epsilon(e_q<\tau_0^{1}, X_{e_q}^{1} \in\mathrm{d}y)}{1- \mathds{E}e^{-q\tau_\epsilon} \mathds{E}_\epsilon e^{-q\tau_0}} \\
   % &= \lim_{\epsilon \rightarrow 0^+}  \frac{\mathds{E}e^{-q\tau_\epsilon} \Big( \mathds{P}_\epsilon(X^{1}_{e_q} \in\mathrm{d}y)- \mathds{P}_\epsilon(e_q>\tau_0^{1}, X^{1}_{e_q} \in\mathrm{d}y) \Big)}{1- \mathds{E}e^{-q\tau_\epsilon} \mathds{E}_\epsilon e^{-q\tau_0}} \\
   &= \lim_{\epsilon \rightarrow 0^+}  \frac{\mathds{E}e^{-q\tau_\epsilon} \Big( \mathds{P}_\epsilon(X^{1}_{e_q} \in \mathrm{d}y)- \mathds{P}_\epsilon(e_q>\tau_0^{1}) \mathds{P}(X^{1}_{e_q} \in \mathrm{d}y) \Big) }{1- \mathds{E}e^{-q\tau_\epsilon} \mathds{E}_\epsilon e^{-q\tau_0}} \\
    &= \lim_{\epsilon \rightarrow 0^+}  \frac{\mathds{E}e^{-q\tau_\epsilon} \Big( \mathds{P}_\epsilon(X^{1}_{e_q} \in \mathrm{d}y)- \mathds{E}_\epsilon e^{-q \tau_0} \mathds{P}(X^{1}_{e_q} \in \mathrm{d}y) \Big) }{1- \mathds{E}e^{-q\tau_\epsilon} \mathds{E}_\epsilon e^{-q\tau_0}}\\
    &= \lim_{\epsilon \rightarrow 0^+} \frac{\mathds{P}_\epsilon(X^{1}_{e_q} \in \mathrm{d}y)-e^{\rho^+_1\epsilon} \mathds{P}(X^{1}_{e_q} \in \mathrm{d}y)}{(1-c_2(q)) e^{\rho^+_2\epsilon} + (c_2(q)-1) e^{\rho^+_1\epsilon}} \hspace{0.4cm} \textit{l'H\^{o}pital's rule} \\
    &=\lim_{\epsilon \rightarrow 0^+} \frac{\frac{d}{d \epsilon}\mathds{P}_\epsilon(X^{1}_{e_q} \in \mathrm{d}y) - \rho^+_1 e^{\rho^+_1 \epsilon} \mathds{P}(X^{1}_{e_q} \in \mathrm{d}y)}{(1-c_2(q))\rho^+_2 e^{\rho^+_2\epsilon} + (c_2(q)-1)\rho^+_1 e^{\rho^+_1\epsilon}} \\ 
    &=\lim_{\epsilon \rightarrow 0^+} \frac{\frac{q}{\delta^+} (-\mu_+ + \delta^+) e^{(\mu_+ - \delta^+)(y-\epsilon)}- \rho^+_1 e^{\rho^+_1 \epsilon} \frac{q}{\delta^+}e^{(\mu_+ - \delta^+)y}}{(1-c_2(q))\rho^+_2 e^{\rho^+_2\epsilon} + (c_2(q)-1)\rho^+_1 e^{\rho^+_1\epsilon}} \mathrm{d} y \\ 
   % &=\frac{\frac{q}{\delta^+} \rho^+_2 e^{(\mu_+ - \delta^+)y}- \rho^+_1  \frac{q}{\delta^+}e^{(\mu_+ - \delta^+)y}}{(1-c_2(q))\rho^+_2  + (c_2(q)-1)\rho^+_1 } \mathrm{d} y \\
    &=\frac{\frac{q}{\delta^+} (\rho^+_2 -\rho^+_1) e^{(\mu_+ - \delta^+)y}}{(1-c_2(q))(\rho^+_2 -\rho^+_1) } \mathrm{d} y \\
   &= \frac{q}{\delta^+(1-c_2(q))}  e^{(\mu_+ - \delta^+)y} \mathrm{d} y,
    \end{split}
\end{equation}
where $\mathds{P}(X_{e_q}^{1} \in \mathrm{d}y) = \frac{q}{\delta^+}e^{(\mu_+ - \delta^+)y} \mathrm{d}y$.
% So, by $\delta^+ = \sqrt{\mu^2_+ + 2q}$,
% \begin{equation} \label{laplace2drists}
%     \int_{0}^{\infty} q e^{-qt}\mathds{P}(X_t\in \mathrm{d}y) \mathrm{d}t = \frac{q}{ \sqrt{\mu^2_+ + 2q}(1-c_2(q))}  e^{(\mu_+ -  \sqrt{\mu^2_+ + 2q})y} \mathrm{d} y.
% \end{equation}
Substituting $\delta^+$ and $c_2(q)$,
\begin{equation} \label{laplace2drists-form2}
\begin{split}
     &\int_{0}^{\infty} q e^{-qt}\mathds{P}(X_t\in \mathrm{d}y) \mathrm{d}t = \\
     &\frac{2(1+\beta)q}{ (1+\beta) \mu_+ - (1-\beta) \mu_- + (1+\beta)\sqrt{\mu_+^2 + 2q} + (1-\beta) \sqrt{\mu_-^2+2q}}  e^{(\mu_+ -  \sqrt{\mu^2_+ + 2q})y} \mathrm{d} y.
\end{split}
\end{equation}

When $y<\epsilon <0$, for $\epsilon<0$. Let $X_t^{2} = \mu_-t + B_t$ be a Brownian motion and let $\tau_x^{2}$ be the hitting time of process $X_t^{2}$. For $y<\epsilon <0$, by the Markov property and the memoryless property,
\begin{equation}
    \begin{split}
        \mathds{P}(X_{e_q} \in \mathrm{d}y) &= \mathds{P}(\tau_\epsilon<e_q, X_{e_q} \in \mathrm{d}y) \\ 
    &= \mathds{P}(\tau_\epsilon<e_q)\mathds{P}_{\epsilon}(X_{e_q} \in \mathrm{d}y) \\
    &= \mathds{E}e^{-q\tau_\epsilon} \Big( \mathds{P}_\epsilon(\tau_0<e_q, X_{e_q} \in \mathrm{d}y) + \mathds{P}_\epsilon(e_q<\tau_0, X_{e_q} \in \mathrm{d}y) \Big) \\ 
    &= \mathds{E}e^{-q\tau_\epsilon} \Big( \mathds{E}_\epsilon e^{-q\tau_0} \mathds{P}(X_{e_q} \in \mathrm{d}y) + \mathds{P}_\epsilon(e_q<\tau_0^{2}, X_{e_q}^{2} \in \mathrm{d}y) \Big).
    \end{split}
\end{equation}
Solving the above equation we have
\begin{equation} \label{limit3_1}
    \begin{split}
   \int_{0}^{\infty} q e^{-qt}\mathds{P}(X_t\in \mathrm{d}y) \mathrm{d}t &= \mathds{P}(X_{e_q}\in \mathrm{d}y) \\
   % &= \frac{\mathds{E}e^{-q\tau_\epsilon} \mathds{P}_\epsilon(e_q<\tau_0^{2}, X_{e_q}^{2} \in \mathrm{d}y)}{1- \mathds{E}e^{-q\tau_\epsilon} \mathds{E}_\epsilon e^{-q\tau_0}} \\
   &= \lim_{\epsilon \rightarrow 0^-} \frac{\mathds{E}e^{-q\tau_\epsilon} \mathds{P}_\epsilon(e_q<\tau_0^{2}, X_{e_q}^{2} \in \mathrm{d}y)}{1- \mathds{E}e^{-q\tau_\epsilon} \mathds{E}_\epsilon e^{-q\tau_0}} \\
   % &= \lim_{\epsilon \rightarrow 0^-}  \frac{\mathds{E}e^{-q\tau_\epsilon} \Big( \mathds{P}_\epsilon(X^{2}_{e_q} \in \mathrm{d}y)- \mathds{P}_\epsilon(e_q>\tau_0^{2}, X^{2}_{e_q} \in \mathrm{d}y) \Big)}{1- \mathds{E}e^{-q\tau_\epsilon} \mathds{E}_\epsilon e^{-q\tau_0}} \\
   &= \lim_{\epsilon \rightarrow 0^-}  \frac{\mathds{E}e^{-q\tau_\epsilon} \Big( \mathds{P}_\epsilon(X^{2}_{e_q} \in \mathrm{d}y)- \mathds{P}_\epsilon(e_q>\tau_0^{2}) \mathds{P}(X^{2}_{e_q} \in \mathrm{d}y) \Big) }{1- \mathds{E}e^{-q\tau_\epsilon} \mathds{E}_\epsilon e^{-q\tau_0}} \\
    &= \lim_{\epsilon \rightarrow 0^-}  \frac{\mathds{E}e^{-q\tau_\epsilon} \Big( \mathds{P}_\epsilon(X^{2}_{e_q} \in \mathrm{d}y)- \mathds{E}_\epsilon e^{-q \tau_0} \mathds{P}(X^{2}_{e_q} \in \mathrm{d}y) \Big) }{1- \mathds{E}e^{-q\tau_\epsilon} \mathds{E}_\epsilon e^{-q\tau_0}}\\
    &= \lim_{\epsilon \rightarrow 0^-} \frac{\mathds{P}_\epsilon(X^{2}_{e_q} \in \mathrm{d}y)-e^{\rho^-_2\epsilon} \mathds{P}(X^{2}_{e_q} \in \mathrm{d}y)}{(c_1(q)-1) e^{\rho^-_2\epsilon} + (1-c_1(q)) e^{\rho^-_1\epsilon}} \hspace{0.4cm} \textit{l'H\^{o}pital's rule} \\
    &=\lim_{\epsilon \rightarrow 0^-} \frac{\frac{d}{d \epsilon}\mathds{P}_\epsilon(X^{2}_{e_q} \in \mathrm{d}y) - \rho^-_2 e^{\rho^-_2 \epsilon} \mathds{P}(X^{2}_{e_q} \in \mathrm{d}y)}{(c_1(q)-1) \rho^-_2e^{\rho^-_2\epsilon} + (1-c_1(q)) \rho^-_1e^{\rho^-_1\epsilon}} \\ 
    &=\lim_{\epsilon \rightarrow 0^-} \frac{\frac{q}{\delta^-} (-\mu_- - \delta^-) e^{(\mu_- + \delta^-)(y-\epsilon)}- \rho^-_2 e^{\rho^-_2 \epsilon} \frac{q}{\delta^-}e^{(\mu_- + \delta^-)y}}{(c_1(q)-1) \rho^-_2e^{\rho^-_2\epsilon} + (1-c_1(q)) \rho^-_1e^{\rho^-_1\epsilon}} \mathrm{d} y\\ 
    % &=\frac{\frac{q}{\delta^-} (-\mu_- - \delta^-) e^{(\mu_- + \delta^-)y}- \rho^-_2 \frac{q}{\delta^-}e^{(\mu_- + \delta^-)y}}{(c_1(q)-1) \rho^-_2 + (1-c_1(q)) \rho^-_1 } \mathrm{d} y\\
    &=\frac{-\frac{q}{\delta^-} (\rho^-_2 -\rho^-_1) e^{(\mu_- + \delta^-)y}}{-(1-c_1(q))(\rho^-_2 -\rho^-_1) } \mathrm{d} y \\
    &= \frac{q}{\delta^-(1-c_1(q))}  e^{(\mu_- + \delta^-)y} \mathrm{d} y,
    \end{split}
\end{equation}
where $\mathds{P}(X_{e_q}^{2} \in \mathrm{d}y) = \frac{q}{\delta^-}e^{(\mu_- + \delta^-)y} \mathrm{d}y$. Substituting $\delta^-$ and $c_1(q)$,
\begin{equation} \label{laplace2drists-form3}
\begin{split}
     &\int_{0}^{\infty} q e^{-qt}\mathds{P}(X_t\in \mathrm{d}y) \mathrm{d}t = \\
     &\frac{2(1-\beta)q}{ (1+\beta) \mu_+ - (1-\beta) \mu_- + (1+\beta)\sqrt{\mu_+^2 + 2q} + (1-\beta) \sqrt{\mu_-^2+2q}}  e^{(\mu_- +  \sqrt{\mu^2_- + 2q})y} \mathrm{d} y.
\end{split}
\end{equation} 
\end{proof}
%By utilizing probability rules and the characteristics of skew Brownian motion, 
We can also identify the potential measure for SBM with arbitrary initial value.
\begin{proposition} \label{prop:pmeasure_x}
    The potential measure of the refracted skew Brownian motion $X_{t}$ with skew level $a=0$ and starting point $x$ is given by:\\
    for $y>0$
     \begin{equation}
\mathds{P}_x(X_{e_q}\in \mathrm{d}y) =
    \begin{cases}
        \frac{q e^{\mu_+ (y-x)}}{\sqrt{\mu_+^2 + 2q}} \Big( e^{- |y-x|\sqrt{\mu_+^2 + 2q}} - e^{- (x+y)\sqrt{\mu_+^2 + 2q}}  \Big) \mathrm{d} y  &  \\
        +\frac{2(1+\beta)q }{(1+\beta) \mu_+ - (1-\beta) \mu_- + (1+\beta) \sqrt{\mu_+^2 + 2q} + (1-\beta)\sqrt{\mu_-^2 + 2q}} e^{\mu_+(y-x) - (x+y)\sqrt{\mu_+^2 + 2q}} \mathrm{d} y,  &  x \geq0, \\
         & \\
        \frac{2(1+\beta)q}{(1+\beta) \mu_+ - (1-\beta) \mu_- + (1+\beta) \sqrt{\mu_+^2 + 2q} + (1-\beta)\sqrt{\mu_-^2 + 2q}} e^{(-\mu_-  + \sqrt{\mu_-^2 + 2q}) x} e^{(\mu_+ - \sqrt{\mu_+^2 + 2q})y} \mathrm{d} y,  &  x \leq 0,
    \end{cases}
\end{equation}
and for $y<0$
\begin{equation}
\mathds{P}_x(X_{ e_{q}}\in \mathrm{d}y) =
    \begin{cases}
        \frac{2(1-\beta)q}{(1+\beta) \mu_+ - (1-\beta) \mu_- + (1+\beta) \sqrt{\mu_+^2 + 2q} + (1-\beta)\sqrt{\mu_-^2 + 2q}} e^{(-\mu_+  - \sqrt{\mu_+^2 + 2q}) x} e^{(\mu_- + \sqrt{\mu_-^2 + 2q})y} \mathrm{d} y,  &  x \geq0, \\
         & \\
       \frac{q e^{\mu_- (y-x)}}{\sqrt{\mu_-^2 + 2q}} \Big( e^{- |y-x|\sqrt{\mu_-^2 + 2q}} - e^{ (x+y)\sqrt{\mu_-^2 + 2q}}  \Big) \mathrm{d} y &  \\
    + \frac{2(1-\beta)q }{(1+\beta) \mu_+ - (1-\beta) \mu_- + (1+\beta) \sqrt{\mu_+^2 + 2q} + (1-\beta)\sqrt{\mu_-^2 + 2q}} e^{\mu_-(y-x) + (x+y)\sqrt{\mu_-^2 + 2q}} \mathrm{d} y.  &  x \leq 0.
    \end{cases}
\end{equation}
\end{proposition}
\begin{proof}
CASE 1: $x \geq 0 $ and $y > 0$,
\begin{eqnarray*}
    \mathds{P}_x(X_{ e_{q}} \in \mathrm{d}y) &=& \mathds{P}_x(X_{ e_{q}} \in \mathrm{d}y,e_{q}<\tau_0) + \mathds{P}_x(X_{ e_{q}} \in \mathrm{d}y,e_{q}>\tau_0) \\
    &=& \mathds{P}_x(X_{ e_{q}}^{1} \in \mathrm{d}y,e_{q}<\tau_0^{1}) + \mathds{P}_x( e_{q}>\tau_0)\mathds{P}(X_{ e_{q}} \in \mathrm{d}y) \\
    &=& \mathds{P}_x(X_{e_{q}}^{1} \in \mathrm{d}y) - \mathds{P}_x(X_{ e_{q}}^{1} \in \mathrm{d}y,e_{q}>\tau_0^{1}) + \mathds{P}_x( e_{q}>\tau_0)\mathds{P}(X_{ e_{q}} \in \mathrm{d}y) \\
    &=& \mathds{P}_x(X_{ e_{q}}^{1} \in \mathrm{d}y) - \mathds{E}_xe^{-q \tau_0^{1}} \mathds{P}(X_{ e_{q}}^{1} \in \mathrm{d}y) + \mathds{E}_xe^{-q \tau_0} \mathds{P}(X_{ e_{q}} \in \mathrm{d}y) \\
    % &=& \frac{q}{\sqrt{\mu_+^2 + 2q}} e^{\mu_+ (y-x) - |y-x|\sqrt{\mu_+^2 + 2q}} \mathrm{d} y - e^{-(\mu_+ + \sqrt{\mu_+^2 + 2q})x} \frac{q}{\sqrt{\mu_+^2 + 2q}} e^{(\mu_+ - \sqrt{\mu_+^2 + 2q})y} \mathrm{d} y \\
    % &+& e^{-(\mu_+ + \sqrt{\mu_+^2 + 2q})x } \frac{q}{\sqrt{\mu_+^2 + 2q} (1-c_2(q))} e^{(\mu_+ - \sqrt{\mu_+^2 + 2q})y} \mathrm{d} y \\
    &=& \frac{q e^{\mu_+ (y-x)}}{\sqrt{\mu_+^2 + 2q}} \Big( e^{- |y-x|\sqrt{\mu_+^2 + 2q}} - e^{- (x+y)\sqrt{\mu_+^2 + 2q}}  \Big) \mathrm{d} y \\
    &+& \frac{2(1+\beta)q }{(1+\beta) \mu_+ - (1-\beta) \mu_- + (1+\beta) \sqrt{\mu_+^2 + 2q} + (1-\beta)\sqrt{\mu_-^2 + 2q}} e^{\mu_+(y-x) - (x+y)\sqrt{\mu_+^2 + 2q}} \mathrm{d} y.
\end{eqnarray*}

CASE 2: $x \leq 0 $ and $y > 0$,
\begin{eqnarray*}
    \mathds{P}_x(X_{ e_{q}} \in \mathrm{d}y) &=& \mathds{P}_x(X_{ e_{q}} \in \mathrm{d}y,e_{q}>\tau_0) \\
    &=& \mathds{P}_x( e_{q}> \tau_0) \mathds{P}(X_{ e_{q}} \in \mathrm{d}y)  \\
    &=& \mathds{E}_xe^{-q \tau_0} \mathds{P}(X_{ e_{q}} \in \mathrm{d}y) \\
    % &=& e^{\rho_2^- x } \frac{q}{\sqrt{\mu_+^2 + 2q} (1-c_2(q))} e^{(\mu_+ - \sqrt{\mu_+^2 + 2q})y} \mathrm{d} y  \\
    &=& \frac{2(1+\beta)q}{(1+\beta) \mu_+ - (1-\beta) \mu_- + (1+\beta) \sqrt{\mu_+^2 + 2q} + (1-\beta)\sqrt{\mu_-^2 + 2q}}\\
    &&\times e^{(-\mu_-  + \sqrt{\mu_-^2 + 2q}) x} e^{(\mu_+ - \sqrt{\mu_+^2 + 2q})y} \mathrm{d} y.
\end{eqnarray*}

CASE 3: $x \geq 0 $ and $y < 0$,
\begin{eqnarray*}
    \mathds{P}_x(X_{ e_{q}} \in \mathrm{d}y) &=& \mathds{P}_x(X_{ e_{q}} \in \mathrm{d}y,e_{q}>\tau_0) \\
    &=& \mathds{P}_x( e_{q}> \tau_0) \mathds{P}(X_{ e_{q}} \in \mathrm{d}y)  \\
    &=& \mathds{E}_xe^{-q \tau_0} \mathds{P}(X_{ e_{q}} \in \mathrm{d}y) \\
    % &=& e^{\rho_1^+ x } \frac{q}{\sqrt{\mu_-^2 + 2q} (1-c_1(q))}e^{(\mu_- + \delta^-)y} \mathrm{d} y \\
    &=& \frac{2(1-\beta)q}{(1+\beta) \mu_+ - (1-\beta) \mu_- + (1+\beta) \sqrt{\mu_+^2 + 2q} + (1-\beta)\sqrt{\mu_-^2 + 2q}}\\
    &&\times e^{(-\mu_+  - \sqrt{\mu_+^2 + 2q}) x} e^{(\mu_- + \sqrt{\mu_-^2 + 2q})y} \mathrm{d} y.
\end{eqnarray*}

CASE 4: $x \leq 0 $ and $y > 0$
\begin{eqnarray*}
    \mathds{P}_x(X_{ e_{q}} \in \mathrm{d}y) &=& \mathds{P}_x(X_{ e_{q}} \in \mathrm{d}y,e_{q}<\tau_0) + \mathds{P}_x(X_{ e_{q}} \in \mathrm{d}y,e_{q}>\tau_0) \\
    &=& \mathds{P}_x(X_{ e_{q}}^{2} \in \mathrm{d}y,e_{q}<\tau_0^{2}) + \mathds{P}_x( e_{q}>\tau_0)\mathds{P}(X_{ e_{q}} \in \mathrm{d}y) \\
    &=& \mathds{P}_x(X_{ e_{q}}^{2} \in \mathrm{d}y) - \mathds{P}_x(X_{ e_{q}}^{2} \in \mathrm{d}y,e_{q}>\tau_0^{2}) + \mathds{P}_x( e_{q}>\tau_0)\mathds{P}(X_{ e_{q}} \in \mathrm{d}y) \\
    &=& \mathds{P}_x(X_{ e_{q}}^{2} \in \mathrm{d}y) - \mathds{E}_xe^{-q \tau_0^{2}} \mathds{P}(X_{ e_{q}}^{2} \in \mathrm{d}y) + \mathds{E}_xe^{-q \tau_0} \mathds{P}(X_{ e_{q}} \in \mathrm{d}y) \\
    % &=& \frac{q}{\sqrt{\mu_-^2 + 2q}} e^{\mu_- (y-x) - |y-x|\sqrt{\mu_-^2 + 2q}} \mathrm{d} y - e^{-(\mu_-  - \sqrt{\mu_-^2 + 2q})x } \frac{q}{\sqrt{\mu_-^2 + 2q}} e^{(\mu_- + \sqrt{\mu_-^2 + 2q})y} \mathrm{d} y \\
    % &+& e^{-(\mu_- - \sqrt{\mu_-^2 + 2q})x } \frac{q}{\sqrt{\mu_-^2 + 2q} (1-c_1(q))} e^{(\mu_- + \sqrt{\mu_-^2 + 2q})y} \mathrm{d} y \\
    &=& \frac{q e^{\mu_- (y-x)}}{\sqrt{\mu_-^2 + 2q}} \Big( e^{- |y-x|\sqrt{\mu_-^2 + 2q}} - e^{ (x+y)\sqrt{\mu_-^2 + 2q}}  \Big) \mathrm{d} y \\
    &+& \frac{2(1-\beta)q }{(1+\beta) \mu_+ - (1-\beta) \mu_- + (1+\beta) \sqrt{\mu_+^2 + 2q} + (1-\beta)\sqrt{\mu_-^2 + 2q}} e^{\mu_-(y-x) + (x+y)\sqrt{\mu_-^2 + 2q}} \mathrm{d} y.
\end{eqnarray*}
\end{proof}

%{\color{blue} Add a remark to explain the typo in\cite{borodin2023distributions}.}

\begin{remark}
    Proposition \ref{prop:pmeasure_x} is obtained  in (1.6) of \cite{borodin2023distributions}. However, note that there is a typo in the first term of (1.6) when $z > 0$ and $x \geq 0$: all terms in the exponential functions should involve $\mu$.
\end{remark}

We further consider the potential measure with barriers imposed.
If $b_-<x<b_+$, then, starting from $x$ with skew level is zero, the probability that $X_{ e_{q}}$ remains between $b_-$ and $b_+$ is:

\begin{proposition}
    Given $b_-<0<b_+, X_0=x, a=0$, the potential measure of the refracted skew Brownian motion with barriers at $b_-$ and $b_+$ is given by
    \[\mathds{P}_x(X_{ e_{q}} \in\mathrm{d}y ,e_{q}<\tau_{b_-} \wedge \tau_{b_+})=\mathds{P}_x(X_{ e_{q}} \in\mathrm{d}y) -\frac{w(x,b_+)}{w(b_-,b_+)} \mathds{P}_{b_-}( X_{ e_{q}} \in\mathrm{d}y) -\frac{w(x,b_-)}{w(b_+,b_-)} \mathds{P}_{b_+}( X_{ e_{q}} \in\mathrm{d}y)\]
    and  remaining below or above level $b_-$ is:
    \[\mathds{P}_x(X_{ e_{q}} \in\mathrm{d}y ,e_{q}<\tau_{b_-} )= \mathds{P}_x(X_{ e_{q}} \in \mathrm{d}y ) - \mathds{E}_x e^{-q \tau_{b_-}} \mathds{P}_{b_-}(X_{ e_{q}} \in\mathrm{d}y).\]
\end{proposition}

\begin{proof}
    
\begin{equation*}
    \begin{split}
        &\mathds{P}_x(X_{ e_{q}} \in\mathrm{d}y ,e_{q}<\tau_{b_-} \wedge \tau_{b_+})\\
        &= \mathds{P}_x(X_{ e_{q}} \in\mathrm{d}y) - \mathds{P}_x(\tau_{b_-}<e_{q} \wedge \tau_{b_+} , X_{ e_{q}} \in\mathrm{d}y)-\mathds{P}_x(\tau_{b_+}<e_{q} \wedge \tau_{b_-} , X_{ e_{q}} \in\mathrm{d}y) \\
        &= \mathds{P}_x(X_{ e_{q}} \in\mathrm{d}y) -\mathds{P}_x(\tau_{b_-}< e_{q} \wedge \tau_{b_+}) \mathds{P}_{b_-}( X_{ e_{q}} \in\mathrm{d}y) -\mathds{P}_x(\tau_{b_+}< e_{q} \wedge \tau_{b_-}) \mathds{P}_{b_+}( X_{ e_{q}} \in\mathrm{d}y)\\
        &=\mathds{P}_x(X_{ e_{q}} \in\mathrm{d}y) -\mathds{E}_x(e^{-q\tau_{b_-}}, \tau_{b_-} < \tau_{b_+}) \mathds{P}_{b_-}( X_{ e_{q}} \in\mathrm{d}y) -\mathds{E}_x(e^{-q\tau_{b_+}}, \tau_{b_+}<\tau_{b_-}) \mathds{P}_{b_+}( X_{ e_{q}} \in\mathrm{d}y) \\
        &= \mathds{P}_x(X_{ e_{q}} \in\mathrm{d}y) -\frac{w(x,b_+)}{w(b_-,b_+)} \mathds{P}_{b_-}( X_{ e_{q}} \in\mathrm{d}y) -\frac{w(x,b_-)}{w(b_+,b_-)} \mathds{P}_{b_+}( X_{ e_{q}} \in\mathrm{d}y).
    \end{split}
\end{equation*}

For the second part: 

\begin{equation*}
    \begin{split}
        \mathds{P}_x(X_{ e_{q}} \in\mathrm{d}y ,e_{q}<\tau_{b_-} ) &= \mathds{P}_x(X_{ e_{q}} \in \mathrm{d}y ) - \mathds{P}_x(X_{ e_{q}} \in\mathrm{d}y ,e_{q} > \tau_{b_-}) \\
        &= \mathds{P}_x(X_{ e_{q}} \in \mathrm{d}y ) - \mathds{P}_x(e_{q} > \tau_{b_-}) \mathds{P}_{b_-}(X_{ e_{q}} \in\mathrm{d}y) \\
        &= \mathds{P}_x(X_{ e_{q}} \in \mathrm{d}y ) - \mathds{E}_x e^{-q \tau_{b_-}} \mathds{P}_{b_-}(X_{ e_{q}} \in\mathrm{d}y).
    \end{split}
\end{equation*}    

\end{proof}
% \begin{equation*}
%     \begin{split}
%         \mathds{P}_x(X_{ e_{q}} \in\mathrm{d}y ,e_{q}< \tau_{b_+}) &= \mathds{P}_x(X_{ e_{q}} \in \mathrm{d}y ) - \mathds{P}_x(X_{ e_{q}} \in\mathrm{d}y ,e_{q} > \tau_{b_+}) \\
%         &= \mathds{P}_x(X_{ e_{q}} \in \mathrm{d}y ) - \mathds{P}_x(e_{q} > \tau_{b_+}) \mathds{P}_{b_+}(X_{ e_{q}} \in\mathrm{d}y) \\
%         &= \mathds{P}_x(X_{ e_{q}} \in \mathrm{d}y ) - \mathds{E}_x e^{-q \tau_{b_+}} \mathds{P}_{b_+}(X_{ e_{q}} \in\mathrm{d}y).
%     \end{split}
% \end{equation*}    

\section{Transition density of RSBM}\label{sec:Laplace_sbm}
In this section, we obtain the transition probability of an RSBM by inverting the Laplace transform.

We first define an auxiliary function 
\begin{equation} \label{func:h}
    h(t;x,\mu) := \frac{|x|}{\sqrt{2 \pi t^3}} e^{\frac{-(x+\mu t)^2}{2t}}, \hspace{1 cm} x,\mu \in \mathds{R}, t>0,
\end{equation}
where its Laplace transform for $x\neq 0$ is given by:
\begin{equation} \label{eq:lap_tr_h}
    \int_{0}^{\infty} e^{-qt} h(t;x,\mu) \mathrm{d}t = \exp\Big[-(\mu + \mathrm{sgn}(x) \sqrt{2q + \mu^2} ) x \Big]
\end{equation}
and zero for $x = 0$. Additionally, we have the following properties for $h$:
\[
h(t;x,\mu) = h(t;-x,-\mu), \hspace{1cm} h(t;x,-\mu)=h(t;-x,\mu) = e^{2\mu x} h(t;x,\mu)
\]
and 
\[
h(t;x_1+x_2, \mu) = \int_{0}^{t} h(t-\tau; x_1, \mu) h(\tau; x_2, \mu) \mathrm{d}\tau,
\]
see \cite{karlin2014first}. 

\begin{theorem} \label{thm:pdf_sbm}
    The density probability of the refracted skew Brownian motion, with skew level $a=0$ and starting point $x=0$ is given by:
    \begin{equation} \label{dist_final_1}
p(t;y) =
    \begin{cases}
        2 (1+ \beta)  e^{2 \mu_+ y} \int_{0}^{\infty} \int_{0}^{t} h(t-\tau; (1+\beta)b + y, \mu_+) h(\tau; (1-\beta)b,-\mu_-) \mathrm{d}\tau \mathrm{d}b,  &  y > 0,\\
         & \\
         2 (1- \beta)  e^{2 \mu_- y} \int_{0}^{\infty} \int_{0}^{t} h(t-\tau; (1+\beta)b , \mu_+) h(\tau; (1-\beta)b-y,-\mu_-) \mathrm{d}\tau \mathrm{d}b,  & y<0.
    \end{cases}
\end{equation}
\end{theorem}
\begin{proof}
 For $y > 0$, taking the Laplace transform for $p(t;y)$ in (\ref{dist_final_1}) we have
\begin{eqnarray*} 
    & &\int_{0}^{\infty}  e^{-qt} \mathds{P}(X_t \in \mathrm{d} y) \mathrm{d}t = \int_{0}^{\infty}  e^{-qt} p(t;y) \mathrm{d}t \mathrm{d}y \\ 
    % &=&  2 (1+ \beta)  e^{2 \mu_+ y} \int_{0}^{\infty}  e^{-qt}  \int_{0}^{\infty} \int_{0}^{t} h(t-\tau; (1+\beta)b + y, \mu_+) h(\tau; (1-\beta)b,-\mu_-) \mathrm{d}\tau \mathrm{d}b \mathrm{d}t \mathrm{d}y \\
    &=& 2 (1+ \beta)  e^{2 \mu_+ y} \int_{0}^{\infty}  \int_{0}^{\infty} e^{-qt} \int_{0}^{t} h(t-\tau; (1+\beta)b + y, \mu_+) h(\tau; (1-\beta)b,-\mu_-) \mathrm{d}\tau \mathrm{d}t \mathrm{d}b \mathrm{d}y \\
    &=& 2 (1+ \beta)  e^{2 \mu_+ y} \int_{0}^{\infty}  \Big[ \int_{0}^{\infty} e^{-qt}  h(t; (1+\beta)b + y, \mu_+) \mathrm{d}t \Big] \Big[ \int_{0}^{\infty} e^{-qt} h(t; (1-\beta)b,-\mu_-) \mathrm{d}t \Big] \mathrm{d}b \mathrm{d}y\\
    &=& 2 (1+ \beta)  e^{2 \mu_+ y} \int_{0}^{\infty}  \Big[ \exp\Big( -(\mu_+ + \sqrt{2q + \mu_+^2})((1+\beta)b + y) \Big) \Big] \Big[ \exp \Big( - (-\mu_- + \sqrt{2q+\mu_-^2}) ((1-\beta) b) \Big) \Big] \mathrm{d}b \mathrm{d}y \\
    &=& 2 (1+ \beta)  e^{2 \mu_+ y} \int_{0}^{\infty} e^{\rho_1^+((1+\beta)b+y)} e^{-\rho_2^-(1-\beta)b} \mathrm{d}b \mathrm{d}y \\
    &=& 2 (1+ \beta)  e^{2 \mu_+ y}  e^{\rho_1^+ y} \int_{0}^{\infty} e^{\rho_1^+(1+\beta)b} e^{-\rho_2^-(1-\beta)b} \mathrm{d}b \mathrm{d}y \\
    &=& 2 (1+ \beta)  e^{-\rho_2^+ y} \int_{0}^{\infty} e^{-( -(1+\beta) \rho_1^+ + (1-\beta) \rho_2^- ) b} \mathrm{d}b \mathrm{d}y \\
    &=& \frac{2 (1+ \beta)}{-(1+\beta) \rho_1^+ + (1-\beta) \rho_2^-} e^{-\rho_2^+ y} \mathrm{d}y \\
    &=& \frac{2 (1+ \beta)}{(1+\beta) [\mu_+ + \sqrt{2q + \mu_+^2}] + (1-\beta) [-\mu_- + \sqrt{2q +\mu_-^2}]} e^{(\mu_+ - \sqrt{2q+\mu_+^2}) y} \mathrm{d}y,
\end{eqnarray*}
which recovers (\ref{eq:pmeasure_x0}).

Similarly, for $y<0$,
\begin{eqnarray*} 
    & &\int_{0}^{\infty}  e^{-qt} \mathds{P}(X_t \in \mathrm{d} y) \mathrm{d}t  = \int_{0}^{\infty}  e^{-qt} p(t;y) \mathrm{d}t \mathrm{d}y \\
    % &=&  2 (1- \beta)  e^{2 \mu_- y} \int_{0}^{\infty}  e^{-qt}  \int_{0}^{\infty} \int_{0}^{t} h(t-\tau; (1+\beta)b , \mu_+) h(\tau; (1-\beta)b-y,-\mu_-) \mathrm{d}\tau \mathrm{d}b \mathrm{d}t \mathrm{d}y\\
    &=& 2 (1- \beta)  e^{2 \mu_- y} \int_{0}^{\infty}  \int_{0}^{\infty} e^{-qt} \int_{0}^{t} h(t-\tau; (1+\beta)b , \mu_+) h(\tau; (1-\beta)b-y,-\mu_-) \mathrm{d}\tau \mathrm{d}t \mathrm{d}b \mathrm{d}y\\
    &=& 2 (1- \beta)  e^{2 \mu_- y} \int_{0}^{\infty}  \Big[ \int_{0}^{\infty} e^{-qt}  h(t; (1+\beta)b , \mu_+) \mathrm{d}t \Big] \Big[ \int_{0}^{\infty} e^{-qt} h(t; (1-\beta)b-y,-\mu_-) \mathrm{d}t \Big] \mathrm{d}b \mathrm{d}y \\
    &=& 2 (1- \beta)  e^{2 \mu_- y} \int_{0}^{\infty}  \Big[ \exp\Big( -(\mu_+ + \sqrt{2q + \mu_+^2})(1+\beta)b \Big) \Big] \Big[ \exp \Big( - (-\mu_- + \sqrt{2q+\mu_-^2)} ((1-\beta) b-y) \Big) \Big] \mathrm{d}b \mathrm{d}y\\
    &=& 2 (1- \beta)  e^{2 \mu_- y} \int_{0}^{\infty} e^{\rho_1^+(1+\beta)b} e^{-\rho_2^-((1-\beta)-y)b} \mathrm{d}b \mathrm{d}y\\
    &=& 2 (1- \beta)  e^{2 \mu_- y}  e^{\rho_2^- y} \int_{0}^{\infty} e^{\rho_1^+(1+\beta)b} e^{-\rho_2^-(1-\beta)b} \mathrm{d}b \mathrm{d}y\\
    &=& 2 (1- \beta)  e^{-\rho_1^- y} \int_{0}^{\infty} e^{-( -(1+\beta) \rho_1^+ + (1-\beta) \rho_2^- ) b} \mathrm{d}b \mathrm{d}y\\
    &=& \frac{2 (1- \beta)}{-(1+\beta) \rho_1^+ + (1-\beta) \rho_2^-} e^{-\rho_1^- y} \mathrm{d}y\\
    &=& \frac{2 (1- \beta)}{(1+\beta) [\mu_+ + \sqrt{2q + \mu_+^2}] + (1-\beta) [-\mu_- + \sqrt{2q +\mu_-^2}]} e^{(\mu_- +  \sqrt{2q+\mu_-^2}) y} \mathrm{d}y,
\end{eqnarray*}
which recovers (\ref{eq:pmeasure_x0}).   
\end{proof}

\begin{remark}
    Our result aligns with the marginal density of $x$ in the joint density provided in equation (2.4) in \cite{gairat2017density}. But it does not agree with equation (2.3) presented in \cite{gairat2017density}, possibly due to a typo in (2.3) for the case  $x>0$. Although the same authors of \cite{gairat2022skew} obtained the density function for the refracted skew Brownian motion containing alternating drifts using equation (2.3), their two results appear to be inconsistent, particularly in more general settings, see Appendix (\ref{app:girat}), for more details.
\end{remark}

The following corollaries concern special cases of Theorem \ref{thm:pdf_sbm}, which appear in the previous work on SBM. 

\begin{corollary} \label{corr:sbm_one_drift}
 The transition probability of a skew Brownian motion with one constant drift, $\mu_- = \mu_+ = \mu$, with skew level $a=0$ and starting point $x=0$ is given by:
    \begin{equation} 
p(t;y) =
    \begin{cases}
        \frac{ (1+ \beta)}{\sqrt{2 \pi t}}   e^{\frac{-(y-\mu t)^2}{2t}}  \Big[1 - \frac{\beta \mu}{2} \sqrt{2\pi t} e^{\frac{(y+\beta \mu t)^2}{2t}} 
 \mathrm{erfc}(\frac{y+\beta \mu t}{\sqrt{2t}}) \Big] \mathrm{d}y,  &  y > 0,\\
         & \\
         \frac{ (1- \beta)}{\sqrt{2 \pi t}}   e^{\frac{-(y-\mu t)^2}{2t}}  \Big[1 - \frac{\beta \mu}{2} \sqrt{2\pi t} e^{\frac{(-y+\beta \mu t)^2}{2t}} 
 \mathrm{erfc}(\frac{-y+\beta \mu t}{\sqrt{2t}}) \Big] \mathrm{d}y,  & y<0,
    \end{cases}
\end{equation}
where $\mathrm{erfc}(z)$ is a complementary error function given by 
\[
\mathrm{erfc}(z) = \frac{2}{\sqrt{\pi}} \int_{z}^{\infty} e^{-t^2} dt.
\]
\end{corollary}

The proof is deferred to the appendix \ref{app:proof_one_drift}.
Note that it recovers the previous result in \cite{appuhamillage2011occupation}. For $\mu = 0$, it corresponds to  a result in \cite{walsh1978diffusion}.

\begin{corollary} \label{corr:sbm_two_alternating_drift}
 The transition probability of a skew Brownian motion with two alternating drifts, $-\mu_- = \mu_+ = \mu$, with skew level $a=0$ and starting point $x=0$ is given by:
    \begin{equation} 
p(t;y) =
    \begin{cases}
        \frac{ (1+ \beta)}{\sqrt{2 \pi t}}  e^{\frac{-(y-\mu t)^2}{2t}} \Big[1 - \frac{ \mu}{2} \sqrt{2\pi t} e^{\frac{(y+ \mu t)^2}{2t}} 
  \mathrm{erfc}(\frac{y+ \mu t}{\sqrt{2t}}) \Big] \mathrm{d}y ,  &  y > 0,\\
         & \\
         \frac{ (1- \beta)}{\sqrt{2 \pi t}}  e^{\frac{-(y+\mu t)^2}{2t}} \Big[1 - \frac{ \mu}{2} \sqrt{2\pi t} e^{\frac{(-y+ \mu t)^2}{2t}} 
  \mathrm{erfc}(\frac{-y+ \mu t}{\sqrt{2t}}) \Big] \mathrm{d}y,  & y<0.
    \end{cases}
\end{equation}
\end{corollary}

Which agrees with \cite{gairat2022skew}. The proof is similar to the previous corollary, so we omit it.

\begin{corollary} \label{cor:laplace_convolution}
Define $f(t;y)$ as the following: 
    \begin{equation*}
    f(t;y) := \begin{cases} 
    \int_0^{\infty} \int_0^t h(t- \tau; (1+\beta)b + y , \mu_+) h(\tau; (1-\beta) b , -\mu_-) \mathrm{d} \tau \mathrm{d} b, & y>0,\\
    \int_0^{\infty} \int_0^t h(t- \tau; (1+\beta)b , \mu_+) h(\tau; (1-\beta) b - y , -\mu_-) \mathrm{d} \tau \mathrm{d} b, & y<0,\\
    \end{cases}
    \end{equation*}
the Laplace transform of $f(t;y)$ is given by:
\begin{equation*}
    \hat{f}(q;y) = \begin{cases} 
    \frac{1}{(1+\beta) [\mu_+ + \sqrt{2q + \mu_+^2}] + (1-\beta) [-\mu_- + \sqrt{2q +\mu_-^2}]}e^{(-\mu_+ - \sqrt{2q + \mu_+^2})y}, & y>0,\\
    \frac{1}{(1+\beta) [\mu_+ + \sqrt{2q + \mu_+^2}] + (1-\beta) [-\mu_- + \sqrt{2q +\mu_-^2}]} e^{(-\mu_- + \sqrt{2q + \mu_-^2})y}, & y<0.\\
    \end{cases} 
\end{equation*}
\end{corollary}
For simplicity we set $f^-(t;y)$ and $\hat{f}^-(q;y)$ to be the functions defined in Corollary \ref{cor:laplace_convolution} for $y<0$, and set $f^+(t;y)$, $\hat{f}^+(q;y)$, when $y>0$.

\begin{proposition} \label{prop:cdf}
     The cumulative distribution function (CDF) of the refracted skew Brownian motion with skew level $a=0$ and initial $x=0$ is given by:     for $z<0$:
     \begin{equation*}
\begin{split}
    \mathds{P}(Z \leq z) &= 2  (1- \beta)  \int_{0}^{\infty} e^{2 \mu_-(1-\beta) b} \Big[ \int_{0}^{t} \frac{1}{\sqrt{2 \pi \tau}} h(t-\tau; (1+\beta)b , \mu_+)  \exp\left( \frac{-((1-\beta)b - z + \mu_- \tau)^2}{2 \tau} \right) \mathrm{d}\tau \\ 
    &- \frac{\mu_-}{2} \int_{0}^{t}   h(t-\tau; (1+\beta)b , \mu_+) \mathrm{erfc} \left( \frac{(1-\beta)b - z + \mu_- \tau}{\sqrt{2 \tau}} \right)   \mathrm{d}\tau \Big] \mathrm{d}b,
    \end{split}
\end{equation*}
and for $z>0$:
\begin{equation*}
    \begin{split}
        \mathds{P}(Z \leq z) &= 1- 2 (1+ \beta) \int_{0}^{\infty}  e^{-2 \mu_+ (1+\beta) b } \Big[ \int_{0}^{t} \frac{1}{\sqrt{2 \pi \tau}}h(t-\tau; (1-\beta)b,-\mu_-)  \exp \left(\frac{-((1+\beta)b + z - \mu_+ \tau)^2}{2 \tau} \right)   \mathrm{d}\tau\\
        &+ \frac{\mu_+}{2} \int_{0}^{t} h(t-\tau; (1-\beta)b,-\mu_-)  \mathrm{erfc} \left( \frac{(1+\beta)b + z - \mu_+ \tau}{\sqrt{2 \tau}} \right)  \mathrm{d}\tau \Big] \mathrm{d} b. 
    \end{split}
\end{equation*}
\end{proposition}

Based on Proposition \ref{prop:pmeasure_x} and Corollary \ref{cor:laplace_convolution} we can derive the transition probability of the refracted skew Brownian motion when the starting point is any real value $x$. 

\begin{proposition} \label{prop:sbm_one_drift}
    The transition probability of the refracted skew Brownian motion, with skew level $a=0$ and starting point $x$ is given by:\\
    for $y>0$
\begin{equation*}
p(t;x,y) =
    \begin{cases}
        \Big[ \phi(y-x;\mu_+) - e^{-2\mu_+ x} \phi(x+y;\mu_+) \Big]     &  \\
        + 2(1+\beta) e^{2\mu_+ y} \int_{0}^{\infty} \int_{0}^{t} h(t-\tau; (1+\beta)b + x + y, \mu_+) h(\tau; (1-\beta)b, -\mu_-)  \mathrm{d} \tau  \mathrm{d} b ,  &  x \geq0, \\
         & \\
        2(1+\beta) e^{2\mu_+ y}  \int_{0}^{\infty} \int_{0}^{t} h(\tau; (1+\beta)b+y, \mu_+)  h(t-\tau; (1-\beta)b -x, -\mu_-) \mathrm{d} \tau \mathrm{d} b , &  x < 0
    \end{cases}
\end{equation*}
and for $y<0$
\begin{equation*}
p(t;x,y) =
    \begin{cases}
        2(1-\beta) e^{2\mu_- y} \int_{0}^{\infty} \int_{0}^{t} h(t - \tau; (1+\beta)b+x, \mu_+) h(\tau; (1-\beta)b-y, -\mu_-)   \mathrm{d} \tau \mathrm{d} b ,  &  x >0, \\
         & \\
       \Big[ \phi(y-x;\mu_-) - e^{2\mu_- y} \phi(-x-y;\mu_-) \Big]   &  \\
    + 2(1-\beta) e^{2 \mu_- y} \int_{0}^{\infty} \int_{0}^{t}  h(\tau; (1+\beta)b, \mu_+)  h(t - \tau; (1-\beta)b-y-x, -\mu_-)   \mathrm{d} \tau \mathrm{d} b, &  x \leq 0,
    \end{cases}
\end{equation*}
where $\phi(z; \mu)$ is the density function of a Brownian motion with drift $\mu$ defined by: 
\[
\phi(z; \mu) = \frac{1}{\sqrt{2 \pi t}} \exp{\frac{-(z-\mu t)^2}{2 t}}\mathrm{d}z.
\]
\end{proposition}
The proof is provided in appendix \ref{app:proof_pdf_general}.

\begin{corollary}
    The transition probability of the refracted Brownian motion with starting point $x$ is given by:
    
For $y>0$,
\begin{equation*}
p(t;x,y) =
    \begin{cases}
        \Big[ \phi(y-x;\mu_+) - e^{-2\mu_+ x} \phi(x+y;\mu_+) \Big]    &  \\
        + 2 e^{2\mu_+ y} \int_{0}^{\infty} \int_{0}^{t} h(t-\tau; b + x + y, \mu_+) h(\tau; b, -\mu_-)  \mathrm{d} \tau  \mathrm{d} b,  &  x \geq 0,\\
         & \\
        2 e^{2\mu_+ y}  \int_{0}^{\infty} \int_{0}^{t} h(\tau; b+y, \mu_+)  h(t-\tau; b -x, -\mu_-) \mathrm{d} \tau \mathrm{d} b, &  x < 0 ,
    \end{cases}
\end{equation*}
and for $y<0$,
\begin{equation*}
p(t;x,y) =
    \begin{cases}
        2 e^{2\mu_- y} \int_{0}^{\infty} \int_{0}^{t} h(t - \tau; b+x, \mu_+) h(\tau; b-y, -\mu_-)   \mathrm{d} \tau \mathrm{d} b ,  &  x >0, \\
         & \\
       \Big[ \phi(y-x;\mu_-) - e^{2\mu_- y} \phi(-x-y;\mu_-) \Big]  &  \\
    + 2 e^{2 \mu_- y} \int_{0}^{\infty} \int_{0}^{t}  h(\tau; b, \mu_+)  h(t - \tau; b-y-x, -\mu_-)   \mathrm{d} \tau \mathrm{d} b, &  x \leq 0.
    \end{cases}
\end{equation*}
\end{corollary}
This recovers the result in \cite{karatzas1984trivariate}.

%{\bf Find the jump size of density at level $a$.}

\begin{proposition}
    For the refracted  skew Brownian motion with $a=0$, the  size of  discontinuity for $p(t;0,y)$ at $y=0$ is given by
    \begin{equation*}
        p(t;0,0^+)-p(t;0,0^-) = 4 \beta \int_{0}^{\infty} \int_{0}^{t} h(t- \tau; (1+\beta) b , \mu_+) h(\tau; (1-\beta)b, -\mu_-) \mathrm{d}\tau \mathrm{d}b
    \end{equation*}
\end{proposition}
 If $\beta = 0$, i.e. there is no skewness in the process, the jump size is zero and the process is continuous. When $\mu_- = \mu_+ = \mu$, then the jump size is:
\[
\frac{\sqrt{2} \beta}{\sqrt{\pi t}} e^{\frac{-\mu^2 t}{2}} \Big( 1- \frac{\beta \mu}{2} \sqrt{2 \pi t} \mathrm{erfc}\Big( \frac{\beta \mu t}{\sqrt{2 t}} \Big) \Big).
\]
When $-\mu_- = \mu_+ = \mu$, i.e. alternating drifts, the jump size is:
\[
\frac{\sqrt{2} \beta}{\sqrt{\pi t}} e^{\frac{-\mu^2 t}{2}} - \beta \mu \mathrm{erfc} \Big( \frac{\mu t}{\sqrt{2 t}} \Big).
\]
In both cases, when $\mu = 0$, the jump size is a linear function of $\beta$:
\[
\frac{\sqrt{2} \beta}{\sqrt{ \pi t}}, 
\hspace{1cm} t>0.
\]
It means that for fixed $t$, when $|\beta|$ is close to $1$, the absolute value of the jump increases. 

Based on alternating case we can argue that when $\mu_- \rightarrow -\infty$ and $\mu_+ \rightarrow \infty $, the jump size tends to $0$, and when $\mu_- \rightarrow \infty$ and $\mu_+ \rightarrow -\infty $, the absolute value of jump size tends to $\infty$.

\section{Long term behaviors} \label{sec:limitdensity}
In this section, we investigate the behavior of an RSBM as $t$ approaches infinity. When $\mu_- > 0$ and $\mu_+ < 0$, as $t$ grows larger and larger, the transition probability of the SBM converges to a smooth density that is independent of the starting point $x$.

\begin{proposition}
Given $\mu_->0$ and $\mu_+<0$, for any $x\in \mathds{R}$ we have
 \[\lim_{t\to\infty}p(t;x,y)=
    \begin{cases} 
    \frac{-2(1+\beta)\mu_+ \mu_-}{(1+\beta)\mu_- - (1-\beta) \mu_+} e^{2\mu_+ y}, & y>0,\\
    \frac{-2(1-\beta)\mu_+ \mu_-}{(1+\beta)\mu_- - (1-\beta) \mu_+} e^{2\mu_- y}, & y<0.\\
    \end{cases}
    \]
\end{proposition}
\begin{proof}
We use the potential measure of SBM to prove this proposition since as $q$ tends to $0$, $t$ tends to infinity. Suppose $\mu_->0$, $\mu_+<0$, from Proposition (\ref{prop:pmeasure_x}),  when $y>0$ and $x \geq 0$, we have 
\begin{equation*} 
\begin{split}
     & \lim_{t\to\infty}p(t;x,y) =  \lim_{q \rightarrow 0} \frac{1}{\mathrm{d}y} \mathds{P}_x(X_{e_q} \in \mathrm{d}y) \\
     & = \lim_{q \rightarrow 0}\frac{q e^{\mu_+ (y-x)}}{\sqrt{\mu_+^2 + 2q}} \Big( e^{- |y-x|\sqrt{\mu_+^2 + 2q}} - e^{- (x+y)\sqrt{\mu_+^2 + 2q}}  \Big) \\
    & + \lim_{q \rightarrow 0} \frac{2(1+\beta)q }{(1+\beta) \mu_+ - (1-\beta) \mu_- + (1+\beta) \sqrt{\mu_+^2 + 2q} + (1-\beta)\sqrt{\mu_-^2 + 2q}} e^{\mu_+(y-x) - (x+y)\sqrt{\mu_+^2 + 2q}} \\
     &=\lim_{q \rightarrow 0} \frac{2(1+\beta) e^{\mu_+(y-x) - (x+y)\sqrt{\mu_+^2 + 2q}} + 2(1+\beta) q \frac{\mathrm{d}}{\mathrm{d} q} \Big( e^{\mu_+(y-x) - (x+y)\sqrt{\mu_+^2 + 2q}} \Big) }{(1+\beta)(\mu^2_+ + 2q)^{\frac{-1}{2}} + (1-\beta) (\mu^2_- + 2q)^{\frac{-1}{2}}} \hspace{1cm} \text{l'H\^{o}pital's rule} \\
     &=\frac{2(1+\beta) e^{\mu_+(y-x) - (x+y)|\mu_+|}}{(1+\beta)|\mu_+|^{-1} + (1-\beta) |\mu_-|^{-1}} \hspace{3cm} |\mu_+|=-\mu_+ \hspace{0.2cm} \text{and} \hspace{0.2cm} |\mu_-|= \mu_- \\
     % &=\frac{2(1+\beta)}{\frac{(1+\beta)}{-\mu_+}+\frac{(1-\beta)}{\mu_-}}e^{2\mu_+ y} \hspace{3cm} |\mu_+|=-\mu_+ \hspace{0.2cm} \text{and} \hspace{0.2cm} |\mu_-|= \mu_- \\
     &=\frac{-2(1+\beta)\mu_+ \mu_-}{(1+\beta)\mu_- - (1-\beta) \mu_+} e^{2\mu_+ y}.
\end{split}
\end{equation*}

When $y>0$ and $x \leq 0$, we have 
\begin{equation*} 
\begin{split}
     & \lim_{t\to\infty}p(t;x,y) =  \lim_{q \rightarrow 0} \frac{1}{\mathrm{d}y} \mathds{P}_x(X_{e_q} \in \mathrm{d}y) \\
     & = \lim_{q \rightarrow 0} \frac{2(1+\beta)q}{(1+\beta) \mu_+ - (1-\beta) \mu_- + (1+\beta) \sqrt{\mu_+^2 + 2q} + (1-\beta)\sqrt{\mu_-^2 + 2q}} e^{(-\mu_-  + \sqrt{\mu_-^2 + 2q}) x} e^{(\mu_+ - \sqrt{\mu_+^2 + 2q})y} \\
     &=\lim_{q \rightarrow 0} \frac{2(1+\beta) e^{(-\mu_-  + \sqrt{\mu_-^2 + 2q}) x} e^{(\mu_+ - \sqrt{\mu_+^2 + 2q})y} + 2(1+\beta) q \frac{\mathrm{d}}{\mathrm{d}q}\Big( e^{(-\mu_-  + \sqrt{\mu_-^2 + 2q}) x} e^{(\mu_+ - \sqrt{\mu_+^2 + 2q})y} \Big) }{(1+\beta)(\mu^2_+ + 2q)^{\frac{-1}{2}} + (1-\beta) (\mu^2_- + 2q)^{\frac{-1}{2}}} \hspace{1cm} \text{l'H\^{o}pital's rule} \\
     &=\frac{2(1+\beta) e^{(-\mu_-  + |\mu_-|) x} e^{(\mu_+ - |\mu_+|)y} }{(1+\beta)|\mu_+|^{-1} + (1-\beta) |\mu_-|^{-1}} \hspace{3cm} |\mu_+|=-\mu_+ \hspace{0.2cm} \text{and} \hspace{0.2cm} |\mu_-|= \mu_-  \\
     % &=\frac{2(1+\beta)}{\frac{(1+\beta)}{-\mu_+}+\frac{(1-\beta)}{\mu_-}}e^{2\mu_+ y} \hspace{3cm} |\mu_+|=-\mu_+ \hspace{0.2cm} \text{and} \hspace{0.2cm} |\mu_-|= \mu_- \\
     &=\frac{-2(1+\beta)\mu_+ \mu_-}{(1+\beta)\mu_- - (1-\beta) \mu_+} e^{2\mu_+ y}.
\end{split}
\end{equation*}
The desired result follows from a property of the Laplace transform.

When $y<0$ and $x \geq 0$, we have 
\begin{equation*} 
\begin{split}
     & \lim_{t\to\infty}p(t;x,y) =  \lim_{q \rightarrow 0} \frac{1}{\mathrm{d}y} \mathds{P}_x(X_{e_q} \in \mathrm{d}y) \\
     & = \lim_{q \rightarrow 0} \frac{2(1-\beta)q}{(1+\beta) \mu_+ - (1-\beta) \mu_- + (1+\beta) \sqrt{\mu_+^2 + 2q} + (1-\beta)\sqrt{\mu_-^2 + 2q}} e^{(-\mu_+  - \sqrt{\mu_+^2 + 2q}) x} e^{(\mu_- + \sqrt{\mu_-^2 + 2q})y} \\
     &=\lim_{q \rightarrow 0} \frac{2(1-\beta) e^{(-\mu_+  - \sqrt{\mu_+^2 + 2q}) x} e^{(\mu_- + \sqrt{\mu_-^2 + 2q})y} + 2(1-\beta) q \frac{\mathrm{d}}{\mathrm{d}q}\Big( e^{(-\mu_+  - \sqrt{\mu_+^2 + 2q}) x} e^{(\mu_- + \sqrt{\mu_-^2 + 2q})y} \Big) }{(1+\beta)(\mu^2_+ + 2q)^{\frac{-1}{2}} + (1-\beta) (\mu^2_- + 2q)^{\frac{-1}{2}}} \hspace{1cm} \text{l'H\^{o}pital's rule} \\
     &=\frac{2(1-\beta) e^{(-\mu_+  - |\mu_+|) x} e^{(\mu_- + |\mu_-|)y} }{(1+\beta)|\mu_+|^{-1} + (1-\beta) |\mu_-|^{-1}} \hspace{3cm} |\mu_+|=-\mu_+ \hspace{0.2cm} \text{and} \hspace{0.2cm} |\mu_-|= \mu_-  \\
     % &=\frac{2(1-\beta)}{\frac{(1+\beta)}{-\mu_+}+\frac{(1-\beta)}{\mu_-}}e^{2\mu_- y} \hspace{3cm} |\mu_+|=-\mu_+ \hspace{0.2cm} \text{and} \hspace{0.2cm} |\mu_-|= \mu_- \\
     &=\frac{-2(1-\beta)\mu_+ \mu_-}{(1+\beta)\mu_- - (1-\beta) \mu_+} e^{2\mu_- y}.
\end{split}
\end{equation*}

When $y<0$ and $x \leq 0$, we have 
\begin{equation*} 
\begin{split}
     & \lim_{t\to\infty}p(t;x,y) =  \lim_{q \rightarrow 0} \frac{1}{\mathrm{d}y} \mathds{P}_x(X_{e_q} \in \mathrm{d}y) \\
     & = \lim_{q \rightarrow 0}\frac{q e^{\mu_- (y-x)}}{\sqrt{\mu_-^2 + 2q}} \Big( e^{- |y-x|\sqrt{\mu_-^2 + 2q}} - e^{ (x+y)\sqrt{\mu_-^2 + 2q}}  \Big) \\
    & + \lim_{q \rightarrow 0}\frac{2(1-\beta)q }{(1+\beta) \mu_+ - (1-\beta) \mu_- + (1+\beta) \sqrt{\mu_+^2 + 2q} + (1-\beta)\sqrt{\mu_-^2 + 2q}} e^{\mu_-(y-x) + (x+y)\sqrt{\mu_-^2 + 2q}} \\
     &=\lim_{q \rightarrow 0} \frac{2(1-\beta) e^{\mu_-(y-x) + (x+y)\sqrt{\mu_-^2 + 2q}} + 2(1-\beta) q \frac{\mathrm{d}}{\mathrm{d}q}\Big( e^{\mu_-(y-x) + (x+y)\sqrt{\mu_-^2 + 2q}} \Big) }{(1+\beta)(\mu^2_+ + 2q)^{\frac{-1}{2}} + (1-\beta) (\mu^2_- + 2q)^{\frac{-1}{2}}} \hspace{1cm} \text{l'H\^{o}pital's rule} \\
     &=\frac{2(1-\beta) e^{\mu_-(y-x) + (x+y)|\mu_-|} }{(1+\beta)|\mu_+|^{-1} + (1-\beta) |\mu_-|^{-1}} \hspace{3cm} |\mu_+|=-\mu_+ \hspace{0.2cm} \text{and} \hspace{0.2cm} |\mu_-|= \mu_-  \\
     % &=\frac{2(1-\beta)}{\frac{(1+\beta)}{-\mu_+}+\frac{(1-\beta)}{\mu_-}}e^{2\mu_- y} \hspace{3cm} |\mu_+|=-\mu_+ \hspace{0.2cm} \text{and} \hspace{0.2cm} |\mu_-|= \mu_- \\
     &=\frac{-2(1-\beta)\mu_+ \mu_-}{(1+\beta)\mu_- - (1-\beta) \mu_+} e^{2\mu_- y}.
\end{split}
\end{equation*}
\end{proof}

%\section{Other properties}
%Write it as a proposition.
For $\mu_+ > 0$ and $\mu_- < 0$, the probability that an SBM, with sufficiently large $t$, tends towards either positive or negative infinity is determined by the following proposition.
\begin{proposition} \label{prop:prob_intfy}
Given $\mu_+>0$ and $\mu_-<0$, for any $x\in \mathds{R}$ we have
\begin{equation*}
\mathds{P}_x(\lim_{t\to\infty}X_t=\infty )=
    \begin{cases}
        \frac{(1+\beta) \mu_+ e^{-2 \mu_- x}}{(1+\beta)\mu_+  - (1-\beta)\mu_- }, &  x < 0, \\
         & \\
         1+\frac{  (1-\beta)\mu_- e^{-2\mu_+ x}}{(1+\beta)\mu_+  - (1-\beta)\mu_- },  &  x > 0,
    \end{cases}
\end{equation*}
and
\begin{equation*}
\mathds{P}_x (\lim_{t\to\infty}X_t=-\infty )=
    \begin{cases}
        1-\frac{ (1+\beta)\mu_+ e^{-2\mu_- x}  }{ (1+\beta) \mu_+ - (1-\beta) \mu_- }, &  x < 0, \\
         & \\
         \frac{-(1-\beta) \mu_- e^{-2\mu_+ x} }{ (1+\beta) \mu_+ - (1-\beta) \mu_- },  &  x > 0.
    \end{cases}
\end{equation*}

\end{proposition}

\begin{proof}
%for $b>0$
%\begin{equation*}
%\mathds{E}_x[e^{-q\tau_b}] =
%    \begin{cases}
%        \frac{e^{\rho^-_2x}}{(1-c_2(q)) e^{\rho^+_2b} + c_2(q) e^{\rho^+_1b}}, &  x < 0 < b, \\
%         & \\
%         \frac{(1-c_2(q)) e^{\rho^+_2x} + c_2(q) e^{\rho^+_1x}}{(1-c_2(q)) e^{\rho^+_2b} + c_2(q) e^{\rho^+_1b}},  & 0 < x < b.
%    \end{cases}
%\end{equation*}
%for $\mu_+>0$ and $\mu_-<0$. 
For $x<0<b$, by Theorem \ref{thm2_zhou} we have
\begin{equation*}
\begin{split}
\mathds{P}_x\left( \sup_{0<t<\infty} X_t \geq b\right)
&=\mathds{P}_x(\tau_b<\infty)=
\lim_{q\to 0+}\mathds{E}_x[e^{-q\tau_b}]\\
& =   \lim_{q \rightarrow 0} \frac{e^{\rho^-_2x}}{(1-c_2(q)) e^{\rho^+_2b} + c_2(q) e^{\rho^+_1b}}\\
&= \frac{(1+\beta) \mu_+ e^{-2 \mu_- x}}{(1+\beta)\mu_+  - (1-\beta)\mu_- +  (1-\beta)\mu_- e^{-2\mu_+b}}.
    \end{split}
\end{equation*}
%for large $b$, 
%\begin{eqnarray*}
%    \mathds{P}_x( \sup_{0<t< e_{q}} X_t \geq b )  =\mathds{P}_x( e_{q} > \tau_b) 
%     = \mathds{E}_xe^{-q \tau_b} %\\
    % &=& \frac{1}{(1-c_2(q)) e^{(-\mu_+ + \sqrt{\mu_+^2 + 2q})b} + c_2(q) e^{(-\mu_+ - \sqrt{\mu_+^2 + 2q})b} }
%\end{eqnarray*}
%and
%\[\mathds{P}_x\left( \sup_{0<t<\infty} X_t \geq b\right)
%=\mathds{P}_x(\tau_b<\infty)=
%\lim_{q\to 0+}\mathds{E}_x[e^{-q\tau_b}],\]
Then
\[
\mathds{P}_x\left( \sup_{0<t< \infty} X_t = \infty \right) 
=\lim_{b\to\infty}\mathds{P}_x\left( \sup_{0<t<\infty} X_t \geq b\right)
=\frac{(1+\beta) \mu_+ e^{-2 \mu_- x}}{(1+\beta)\mu_+  - (1-\beta)\mu_- }.
\] 

Similarly, for $0<x<b$,
\begin{equation*}
\begin{split}
\mathds{P}_x\left( \sup_{0<t<\infty} X_t \geq b\right)
&=\mathds{P}_x(\tau_b<\infty)=
\lim_{q\to 0+}\mathds{E}_x[e^{-q\tau_b}]\\
& =  \lim_{q \rightarrow 0} \frac{(1-c_2(q)) e^{\rho^+_2x} + c_2(q) e^{\rho^+_1x}}{(1-c_2(q)) e^{\rho^+_2b} + c_2(q) e^{\rho^+_1b}} \\
&= \frac{(1+\beta)\mu_+  - (1-\beta)\mu_- +  (1-\beta)\mu_- e^{-2\mu_+ x}}{(1+\beta)\mu_+  - (1-\beta)\mu_- +  (1-\beta)\mu_- e^{-2\mu_+b}}.
    \end{split}
\end{equation*}
Then
\[
\mathds{P}_x\left( \sup_{0<t< \infty} X_t = \infty \right) 
=\lim_{b\to\infty}\mathds{P}_x\left( \sup_{0<t<\infty} X_t \geq b\right)
=\frac{(1+\beta)\mu_+  - (1-\beta)\mu_- +  (1-\beta)\mu_- e^{-2\mu_+ x}}{(1+\beta)\mu_+  - (1-\beta)\mu_- }.
\]

%By Theorem \ref{thm2_zhou} for $b<0$:

%For $b<0$,
%\begin{equation*}
%\mathds{E}_x[e^{-q\tau_b}] =
 %   \begin{cases}
%        \frac{e^{\rho^+_1x}}{c_1(q) e^{\rho^-_2b} + (1-c_1(q)) e^{\rho^-_1b}}, &  b \leq 0 \leq x ,\\
%         & \\
 %        \frac{c_1(q) e^{\rho^-_2x} + (1-c_1(q)) e^{\rho^-_1x}}{c_1(q) e^{\rho^-_2b} + (1-c_1(q)) e^{\rho^-_1b}},  & b \leq x \leq 0.
%    \end{cases}
%\end{equation*}

For $b \leq 0 \leq x$, 
\begin{equation*}
\begin{split}
\mathds{P}_x\left( \inf_{0<t<\infty} X_t \leq b\right)
&=\mathds{P}_x(\tau_b<\infty)=
\lim_{q\to 0-}\mathds{E}_x[e^{-q\tau_b}]\\
& =  \lim_{q \rightarrow 0} \frac{e^{\rho^+_1x}}{c_1(q) e^{\rho^-_2b} + (1-c_1(q)) e^{\rho^-_1b}} \\
&= \frac{-(1-\beta) \mu_- e^{-2\mu_+ x} }{-(1+\beta)\mu_+ e^{-2\mu_- b} + (1+\beta) \mu_+ - (1-\beta) \mu_- }.
    \end{split}
\end{equation*}
It follows that
\[
\mathds{P}_x\left( \inf_{0<t< \infty} X_t = -\infty \right) 
=\lim_{b\to - \infty}\mathds{P}_x\left( \inf_{0<t<\infty} X_t \leq b\right)
= \frac{-(1-\beta) \mu_- e^{-2\mu_+ x} }{ (1+\beta) \mu_+ - (1-\beta) \mu_- }.
\]

For $b \leq x \leq 0$, 
\begin{equation*}
\begin{split}
\mathds{P}_x\left( \inf_{0<t<\infty} X_t \leq b\right)
&=\mathds{P}_x(\tau_b<\infty)=
\lim_{q\to 0^-}\mathds{E}_x[e^{-q\tau_b}]\\
& = \lim_{q \rightarrow 0} \frac{c_1(q) e^{\rho^-_2x} + (1-c_1(q)) e^{\rho^-_1x}}{c_1(q) e^{\rho^-_2b} + (1-c_1(q)) e^{\rho^-_1b}} \\
&= \frac{-(1+\beta)\mu_+ e^{-2\mu_- x} + (1+\beta) \mu_+ - (1-\beta) \mu_- }{-(1+\beta)\mu_+ e^{-2\mu_- b} + (1+\beta) \mu_+ - (1-\beta) \mu_- }.
    \end{split}
\end{equation*}
So, 
\[
\mathds{P}_x\left( \inf_{0<t< \infty} X_t = -\infty \right) 
=\lim_{b\to - \infty}\mathds{P}_x\left( \inf_{0<t<\infty} X_t \leq b\right)
= \frac{ (1+\beta) \mu_+ - (1-\beta) \mu_- - (1+\beta)\mu_+ e^{-2\mu_- x}  }{ (1+\beta) \mu_+ - (1-\beta) \mu_- }.
\]
\end{proof}

\begin{proposition}
    For $\mu_- >0$, $\mu_+<0$ and  any real value $x, z$, we have $\mathds{P}_x(\tau_z < \infty) =1$. 
\end{proposition}
\begin{proof}
    We follow an argument similar to that in Proposition (\ref{prop:prob_intfy}) but for $\mu_->0$ and $\mu_+<0$.
    For $x<0<z$,
    \begin{equation*}
        \mathds{P}_x(\tau_z<\infty) =
        \lim_{q\to 0+}\mathds{E}_x[e^{-q\tau_z}] =  \lim_{q \rightarrow 0} \frac{e^{\rho^-_2x}}{(1-c_2(q)) e^{\rho^+_2z} + c_2(q) e^{\rho^+_1z}} = 1;
    \end{equation*}
for $0<x<z$,
\begin{equation*}
\mathds{P}_x(\tau_z<\infty) =
\lim_{q\to 0+}\mathds{E}_x[e^{-q\tau_z}]=  \lim_{q \rightarrow 0} \frac{(1-c_2(q)) e^{\rho^+_2x} + c_2(q) e^{\rho^+_1x}}{(1-c_2(q)) e^{\rho^+_2z} + c_2(q) e^{\rho^+_1z}} = 1;
\end{equation*}
for $z \leq 0 \leq x$, 
\begin{equation*}
\mathds{P}_x(\tau_z<\infty) =
\lim_{q\to 0-}\mathds{E}_x[e^{-q\tau_z}]  =  \lim_{q \rightarrow 0} \frac{e^{\rho^+_1x}}{c_1(q) e^{\rho^-_2z} + (1-c_1(q)) e^{\rho^-_1z}} =1;
\end{equation*}  
and for $z \leq x \leq 0$, 
\begin{equation*}
\mathds{P}_x(\tau_z<\infty)=
\lim_{q\to 0-}\mathds{E}_x[e^{-q\tau_z}] = \lim_{q \rightarrow 0} \frac{c_1(q) e^{\rho^-_2x} + (1-c_1(q)) e^{\rho^-_1x}}{c_1(q) e^{\rho^-_2z} + (1-c_1(q)) e^{\rho^-_1z}} = 1.
\end{equation*}
Note that for $\mu_->0$ and $\mu_+<0$, as $q \rightarrow 0$ both $c_1(q)$ and $c_2(q)$ tend to $1$,  and $\rho_1^-$, $\rho_2^-$, $\rho_1^+$, $\rho_2^+$ tends to $-2\mu_-$, $0$, $0$, $-2\mu_+$, respectively.
\end{proof}

\begin{proposition}
    Given any real value $x, z$, $\mathds{E}_x(\tau_z )  < \infty$ for $\mu_- >0$ and $\mu_+<0$, and it is infinite if either $\mu_+ =0$ or $\mu_- = 0$. 
\end{proposition}
\begin{proof}
    According to Theorem (\ref{thm2_zhou}), if we define  $ \psi(q) = \mathds{E}_x[e^{-q \tau_z}]$, then $\mathds{E}_x[\tau_z] = -\psi^{\prime}(0)$. Given $\mu_->0$ and $\mu_+<0$, when $x<0<z$,
\begin{equation*}
    \mathds{E}_x[\tau_z] = \frac{ 2 (1+ \beta) (\mu_- \mu_+ z - \mu_+^2 x) + \left( \mu_- (1+\beta) - \mu_+ (1-\beta) \right) \left( e^{-2 \mu_+ z} -1 \right) }{2 \mu_{-} \mu_{+}^{2} (1+ \beta)};
\end{equation*}
when $0<x<z$,
\begin{equation*}
    \mathds{E}_x[\tau_z] = \frac{ 2( 1+\beta) \mu_- \mu_+ (z-x) + \left( \mu_- (1+\beta) - \mu_+ (1 -\beta) \right) \left( e^{-2 \mu_+ z} - e^{-2 \mu_+ x} \right) }{2 \mu_{-} \mu_{+}^{2} (1+ \beta)};
    \end{equation*}
when $z \leq 0 \leq x$,
\begin{equation*}
    \mathds{E}_x[\tau_z] = \frac{ 2(1-\beta) (\mu_- \mu_+ z -  \mu_-^2 x) + \left( \mu_- (1+\beta)  - \mu_+  (1-\beta) \right) \left( 1- e^{-2 \mu_- z} \right) }{2 \mu_{-}^2 \mu_{+} (1- \beta)}
    \end{equation*}
and when $z \leq x \leq 0$, 
\begin{equation*}
    \mathds{E}_x[\tau_z] = \frac{ 2 \mu_- \mu_+ (1- \beta) (z-x) + \left( \mu_- (1+\beta) - \mu_+ (1-\beta) \right) \left( e^{-2 \mu_- x} - e^{-2 \mu_- z} \right) }{2 \mu_{-}^2 \mu_{+} (1- \beta)}.
\end{equation*}
The desired results then follow.
%As it can be seen $\mathds{E}_x[\tau_z]$ is finite, and if either $\mu_+ =0$ or $\mu_- = 0$, it is unbounded. 
\end{proof}

\section{Approximation and Simulation} \label{sec:approx_simu}
As mentioned in the introduction, skew Brownian motion has many applications across various fields. A key aspect of these applications is the generation of samples from the distribution, which can be utilized in Monte Carlo simulations. However, due to the complexity of the stochastic differential equation (SDE) governing SBM and the intricate nature of its cumulative distribution function (CDF), traditional methods such as the Euler scheme or transformation methods are challenging to implement.

Although SBM does not have the independent increments because it is biased at the skew level, we can approximate its CDF on either side of the skew level; here is zero. In this section, we introduce two methods to generate quasi-random samples of skew Brownian motion. The first method approximates the CDF of SBM using a truncated normal distribution and then generates random samples based on the inverse of this approximated CDF. The second method employs a deep neural network to approximate the inverse of the CDF of SBM, subsequently generating random samples using the trained network. 

\subsection{Using Truncated Normal Distribution}
A truncated normal distribution is a modification of the normal (Gaussian) distribution where values below a certain lower bound and/or above a certain upper bound are excluded, or "truncated." This results in a distribution that is confined to a specified range. The truncation can occur on either one side (left or right) or on both sides of the distribution. Suppose $X$ is a normal random variable with mean $\mu$ and variance $\sigma^2$, with probability density function $f(x)$ and cumulative distribution function $F(x)$. Then the PDF, $k(x)$, and CDF, $K(x)$, of the normal distribution truncated in $[a, b]$ is given by:
\[
k(x) = \frac{f(x)}{F(b) - F(a)}, \hspace{1cm} K(x) = \frac{F(x) - F(a)}{F(b) - F(a)}.
\]

Suppose $X = (X_t)_{t\geq 0}$ is the refracted skew Brownian motion, with skew level $a=0$ and starting point $x=0$ with the CDF given in Proposition (\ref{prop:cdf}). Given that the density function of $X$ on either side of the skew level closely resembles a normal distribution, we assume that $X$ is a truncated Brownian motion. So we approximate the CDF of RSBM with the following mixture of truncated normal distributions:
\begin{equation*}
    H(x) =
    \begin{cases}
        \alpha K_1(x), & x <0, \\
        \alpha + (1- \alpha) K_2(x),  & x>0,
    \end{cases}
\end{equation*}
where $K_1(x) = \frac{F(x)}{F(0)}$ is the CDF of a truncated normal variable in support $(-\infty, 0)$, and $K_2(x) = \frac{G(x) - G(0)}{1-G(0)}$ is the CDF of a truncated normal variable in support $(0,\infty)$.  Here, $F$ is a normal CDF with mean $\mu_1$, variance $\sigma_1^2$ and $G$ is the normal CDF with mean $\mu_2$, variance $\sigma_2^2$, and $\alpha$ is a constant in $[0,1]$.

Given the parameters $\alpha, \mu_1, \sigma_1^2, \mu_2, \sigma_2^2$ of $H$, we can generate random samples of the RSBM by transforming uniform random variables using the inverse CDF method. For a uniform random variable $u$, the quasi-random sample RSBM is given by:
\begin{equation} \label{eq:tna_sample}
    \hat{X} =
    \begin{cases}
        F^{-1} \left( \frac{u}{\alpha} F(0) \right), & u \leq \alpha, \\
        G^{-1}\left( \frac{u-\alpha}{1-\alpha} \Big(1 - G(0) \Big) + G(0) \right),  & u > \alpha.
    \end{cases}
\end{equation}
The remaining task is to determine the parameters of $H$. To achieve an accurate approximation, we use the Sequential Least Squares Programming (SLSQP) method \cite{kraft1988software} in the SciPy Python package. This optimization technique is employed to minimize the quantile of absolute errors between the RSBM CDF and $H$. Algorithm (\ref{alg:sbm_tna}) outlines the necessary steps to generate random samples of RSBM using a mixture truncated normal distribution.

\begin{algorithm}  
  \caption{RSBM Quasi-Random Sample Generation Using Truncated Normal Distribution}  
  \label{alg:sbm_tna} 
\begin{algorithmic}
    \Require $\mathbf{x}$: Vector of equispaced or non-equispaced values in the interval $[-A, A]$ with length $N$, \\ $params$: Skew Brownian Motion parameters 
    \Ensure $\hat{X}$: Random sample of RSBM
    \State Calculate $\mathds{P}(X \leq x_i)$ for each $i \in \{1,\dots, N\}$ based on Proposition (\ref{prop:cdf}),
    \State Estimate parameters $\mu_1, \sigma_1, \mu_2, \sigma_2, \alpha$ of the mixture truncated normal distribution by:
    $$\min_i Q_{0.99}\left(|\mathds{P}(X \leq x_i) - H(x_i)|\right),$$
    \State Generate uniform random variables $u$,
    \State Generate RSBM random samples $\hat{X}$ using the approximated mixture truncated normal distribution as per equation (\ref{eq:tna_sample}).
\end{algorithmic}
\end{algorithm}

\subsection{Using Deep Neural Network}
A deep neural network (DNN) is a type of artificial neural network (ANN) with multiple layers between the input and output layers. These additional layers, often referred to as hidden layers, enable the network to model complex, hierarchical patterns in data. DNNs are a core component of deep learning, a subset of machine learning, and have many applications in computer vision, finance, games, etc. Researchers have shared an immense interest in applying neural networks to generate quasi-random samples from a distribution. For instance, Horger et al. in \cite{horger2018deep} introduced a neural network model capable of transforming samples from a uniform distribution into samples of any specified probability density function. Hofert et al. in \cite{hofert2021quasi} presented a new approach for generating quasi-random samples from multivariate distributions with any underlying copula using generative neural networks. Furthermore, Gibson et al. in \cite{gibson2023flow} proposed a method that employs Normalizing Flow generative models combined with Importance Sampling to simulate samples from rare event distributions and achieve accurate estimates. Here, we use DNN to approximate the inverse of CDF of an RSBM, then generate quasi-random samples using the trained network. Since deep neural networks (DNNs) are not the main focus of this paper, we will not delve into their details. Instead, we will briefly explain our algorithm and leave the rest to readers to explore further as needed.

Algorithm (\ref{alg:sbm_dnn}) explains the steps for training the neural network and generating random sample variables of RSBM. We use one node in both the input and output layers and employ four hidden layers with 512, 1025, 512, and 1024 nodes, respectively. Additionally, we utilize the Adam optimization algorithm. The number of epochs and iterations are 10,000 and 250 respectively.

\begin{algorithm}  
  \caption{RSBM Quasi-Random Sample Generation Using DNN}  
  \label{alg:sbm_dnn} 
\begin{algorithmic}
    \Require  $z$: Equispaced or non-equispaced values in the interval $[-A, A]$, \\$params$: Skew Brownian Motion parameters 
    \Ensure  $\hat{Z}$: Random sample of RSBM
\State Calculate $F(z)= \mathds{P}(Z\leq z)$ based on Proposition (\ref{prop:cdf}).
\For {epochs}
    \For {Iterations}
        \State Get a chunk of $F(z)$
        \State Calculate the $\hat{z}=\hat{F}^{-1}(F(z))$ using the neural network 
        \State Calculate the objective function (Means Absolute Error = $\sum_i |z_i - \hat{z}_i|$)
        \State Update the neural network weights using back-propagation and auto-gradient
    \EndFor
\EndFor
\State Generate uniform random variables $u$
\State Generate RSBM random samples $\hat{Z}$ using trained neural network and $u$ as its input.
\end{algorithmic}
\end{algorithm}

\subsection{Simulation Results}
In this section, we generate random samples of RSBM with different parameters using the two above-mentioned methods. In order to test the goodness-of-fit of generated random models, we employ a one-sample Kolmogorov–Smirnov (K-S) test. KS test is a nonparametric test to determine whether a sample comes from a specified continuous distribution. The Kolmogorov–Smirnov statistic quantifies the distance between the empirical distribution function of the sample and the cumulative distribution function of the reference distribution. The null distribution of this statistic is calculated under the null hypothesis that the sample is drawn from the reference distribution. Since the RSBM CDF is a right-continuous distribution with a single jump at the skew level, we can use a one-sample K-S test to determine whether the generated quasi-random samples follow the same distribution as the RSBM.

We consider four models for RSBM with different parameters (Table \ref{tab:sbm_models}) and utilize the approximation methods to generate random samples. In addition, we consider $A = 20$ in both Algorithms (\ref{alg:sbm_tna} and \ref{alg:sbm_dnn}). 
\begin{table}[h]
    \centering
    \begin{tabular}{l|c c c c}
         Models & $\mu_-$ & $\mu_+$ &$t$& $\beta$  \\
         \hline
         Model 1 & -0.1 & 0.1 & 2 & 0.3 \\
         Model 2 & 2 & -4 & 2 & 0.7 \\
         Model 3 & 1 & 3 & 1 & -0.1 \\
         Model 4 & -2 & -3 & 3 &  0.9
    \end{tabular}
    \caption{RSBM models with different parameters}
    \label{tab:sbm_models}
\end{table}

Figures (\ref{fig:NN_cdf_pdf}) and (\ref{fig:TN_cdf_pdf}) present the random samples of RSBM generated using Deep Neural Network (DNN) and Truncated Normal Approximations (TNA), respectively. The first row of each figure compares the histogram of generated quasi-random samples with the theoretical PDF of RSBM, while the second row illustrates the empirical CDF of the generated samples against the theoretical CDF of the process. Note that the theoretical PDF and CDF are approximated using the Quadrature Adaptive General Segment (QAGS) algorithm as proposed in \cite{piessens2012quadpack}, which is available in the SciPy package in Python.
 
It can be observed that the generated samples are compatible with the theoretical distribution. Furthermore, the Kolmogorov-Smirnov (K-S) test indicates that the generated samples align well with the RSBM with a 95\% confidence level. Table (\ref{tab:kstest}) displays the results of the K-S test for both approximation methods across different models, showing the effectiveness of each method in accurately generating samples that reflect the characteristics of the theoretical RSBM distribution.

\begin{table}[h]
\centering
\begin{tabular}{|c|c|c|c|}
\hline
Models & Method & K-S stats & p\_value \\ \hline
\multirow{2}{*}{Model 1} & DNN & 0.0143 & 0.2559 \\ \cline{2-4} 
        & TNA & 0.0088 & 0.8305 \\ \hline
\multirow{2}{*}{Model 2} & DNN & 0.0116 & 0.5105 \\ \cline{2-4} 
        & TNA & 0.0123 & 0.4277 \\ \hline
\multirow{2}{*}{Model 3} & DNN & 0.0106 & 0.6226 \\ \cline{2-4} 
        & TNA & 0.0148 & 0.2181 \\ \hline
\multirow{2}{*}{Model 4} & DNN & 0.0096 & 0.7366 \\ \cline{2-4} 
        & TNA & 0.0138 & 0.2874 \\ \hline
\end{tabular}
\caption{K-S Test Results for DNN and TNA Methods Across Different Models}
\label{tab:kstest}
\end{table}

\begin{figure}[h] 
    \centering
    \begin{subfigure}[b]{0.24\textwidth}
        \centering
        \includegraphics[width=\textwidth]{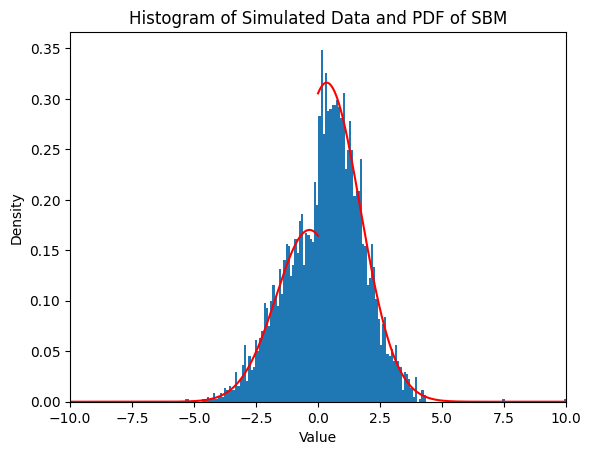}
        % \caption{Model 1}
    \end{subfigure}
    \hfill
    \begin{subfigure}[b]{0.24\textwidth}
        \centering
        \includegraphics[width=\textwidth]{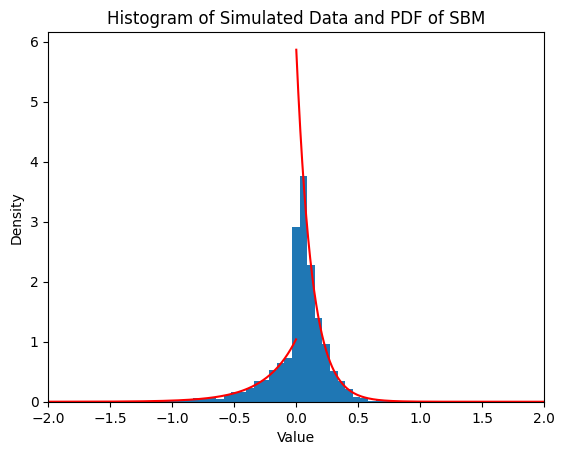}
        % \caption{Model 2}
    \end{subfigure}
    \hfill
    \begin{subfigure}[b]{0.24\textwidth}
        \centering
        \includegraphics[width=\textwidth]{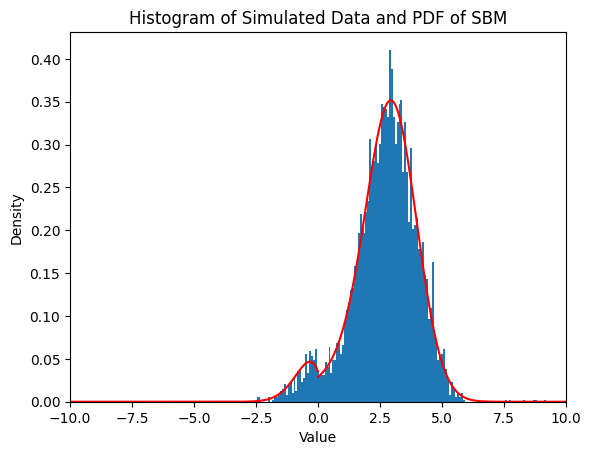}
        % \caption{Model 3}
    \end{subfigure}
    \hfill
    \begin{subfigure}[b]{0.24\textwidth}
        \centering
        \includegraphics[width=\textwidth]{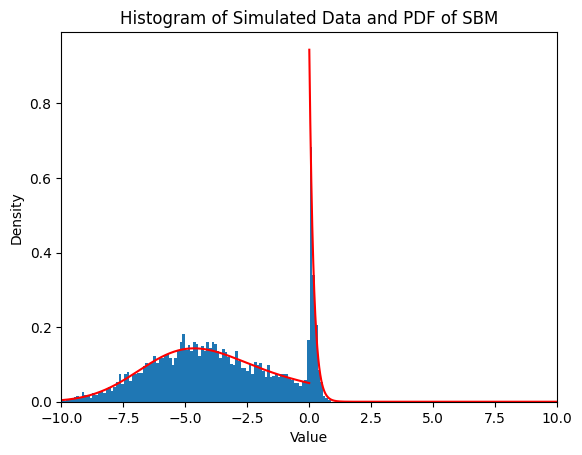}
        % \caption{Model 4}
    \end{subfigure}
    \vskip\baselineskip
    \begin{subfigure}[b]{0.24\textwidth}
        \centering
        \includegraphics[width=\textwidth]{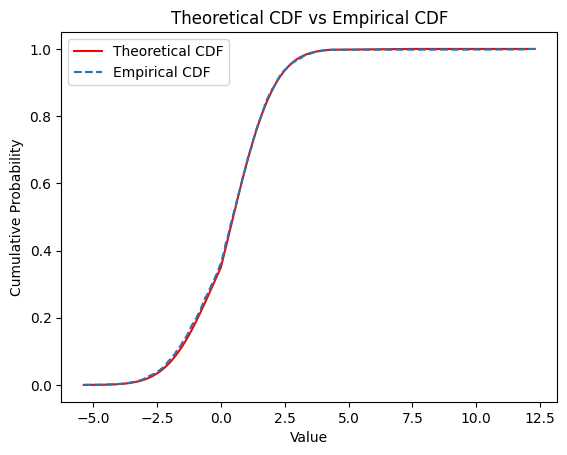}
        \caption{Model 1}
    \end{subfigure}
    \hfill
    \begin{subfigure}[b]{0.24\textwidth}
        \centering
        \includegraphics[width=\textwidth]{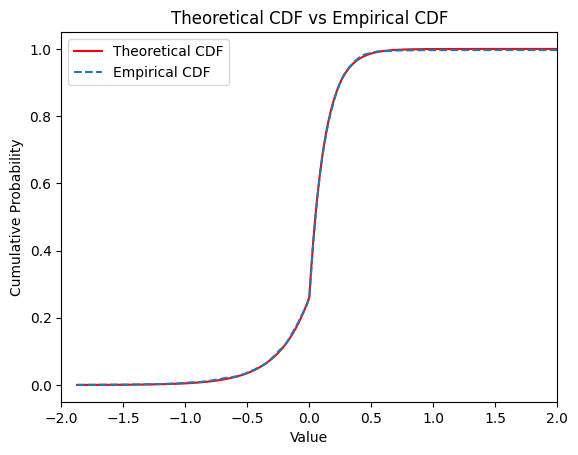}
        \caption{Model 2}
    \end{subfigure}
    \hfill
    \begin{subfigure}[b]{0.24\textwidth}
        \centering
        \includegraphics[width=\textwidth]{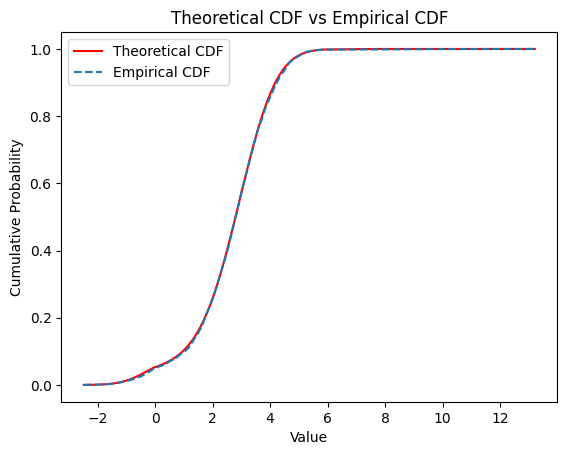}
        \caption{Model 3}
    \end{subfigure}
    \hfill
    \begin{subfigure}[b]{0.24\textwidth}
        \centering
        \includegraphics[width=\textwidth]{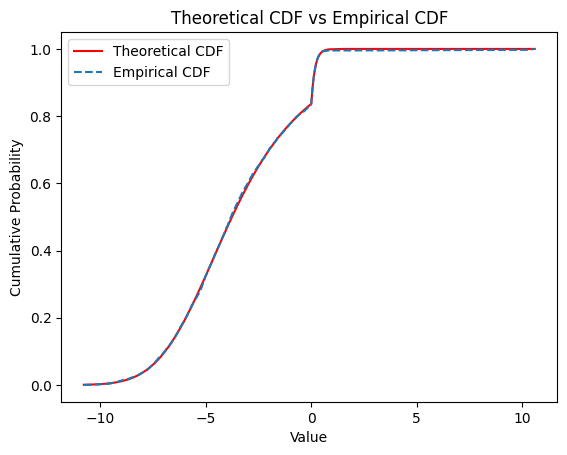}
        \caption{Model 4}
    \end{subfigure}
    \caption{Histogram and CDF of generated random samples for different models with deep neural network against the exact PDF and CDF of the model}
    \label{fig:NN_cdf_pdf}
\end{figure}

\begin{figure}[h] 
    \centering
    \begin{subfigure}[b]{0.24\textwidth}
        \centering
        \includegraphics[width=\textwidth]{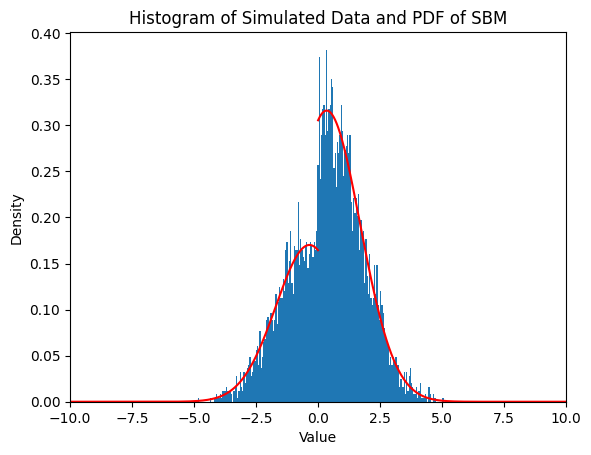}
        % \caption{Model 1}
    \end{subfigure}
    \hfill
    \begin{subfigure}[b]{0.24\textwidth}
        \centering
        \includegraphics[width=\textwidth]{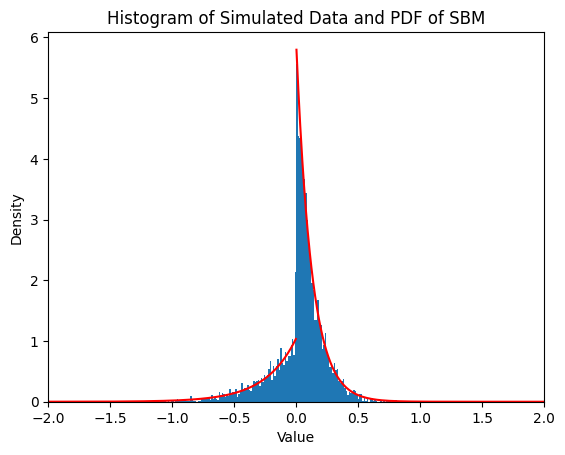}
        % \caption{Model 2}
    \end{subfigure}
    \hfill
    \begin{subfigure}[b]{0.24\textwidth}
        \centering
        \includegraphics[width=\textwidth]{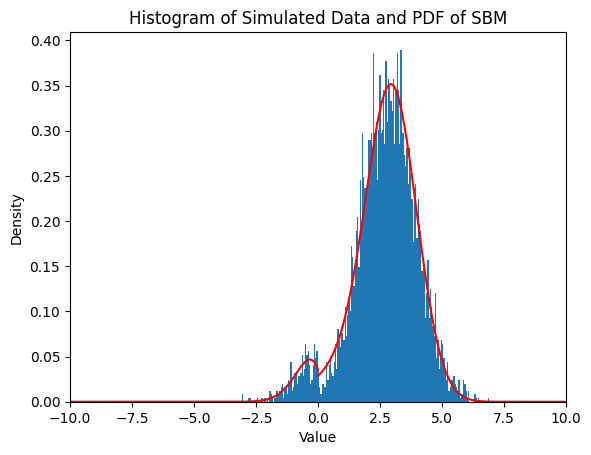}
        % \caption{Model 3}
    \end{subfigure}
    \hfill
    \begin{subfigure}[b]{0.24\textwidth}
        \centering
        \includegraphics[width=\textwidth]{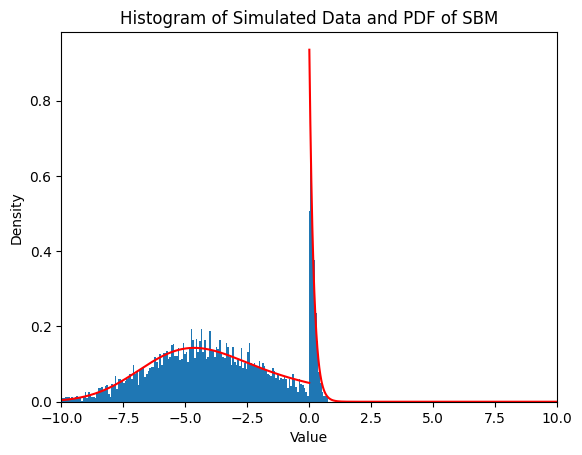}
        % \caption{Model 4}
    \end{subfigure}
    \vskip\baselineskip
    \begin{subfigure}[b]{0.24\textwidth}
        \centering
        \includegraphics[width=\textwidth]{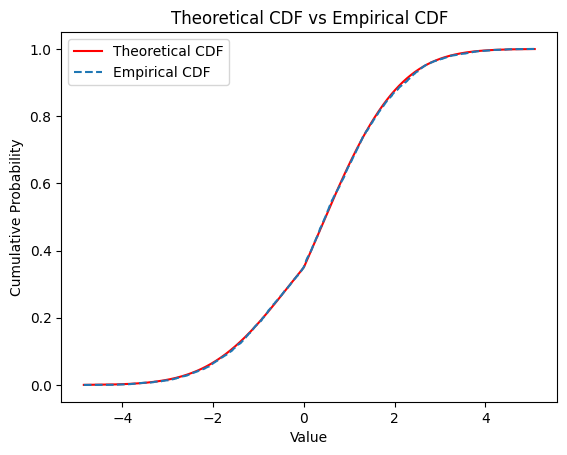}
        \caption{Model 1}
    \end{subfigure}
    \hfill
    \begin{subfigure}[b]{0.24\textwidth}
        \centering
        \includegraphics[width=\textwidth]{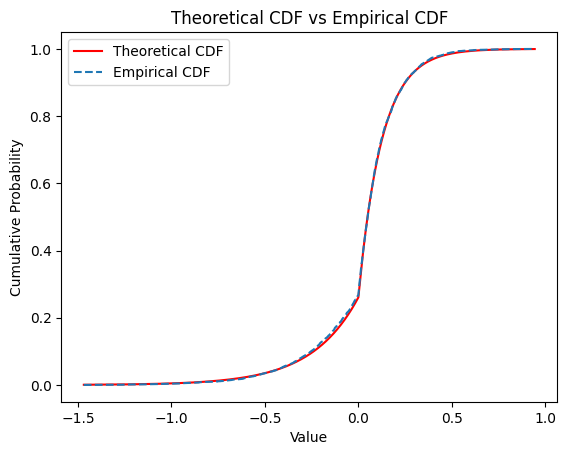}
        \caption{Model 2}
    \end{subfigure}
    \hfill
    \begin{subfigure}[b]{0.24\textwidth}
        \centering
        \includegraphics[width=\textwidth]{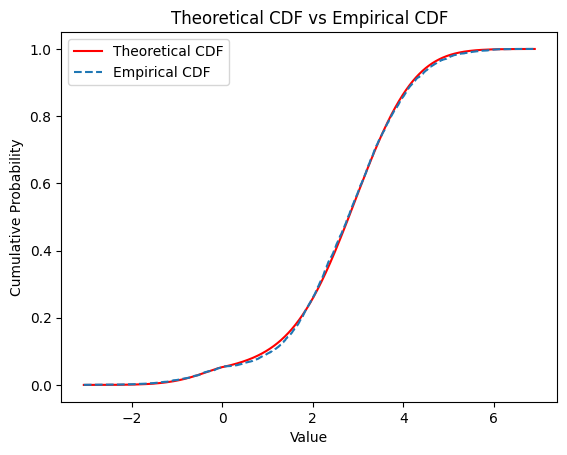}
        \caption{Model 3}
    \end{subfigure}
    \hfill
    \begin{subfigure}[b]{0.24\textwidth}
        \centering
        \includegraphics[width=\textwidth]{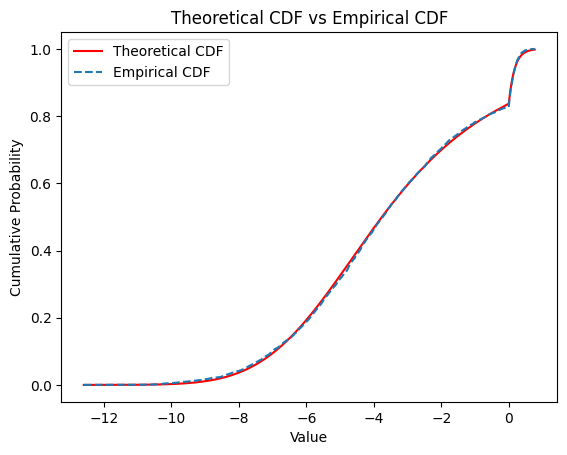}
        \caption{Model 4}
    \end{subfigure}
    \caption{Histogram and CDF of generated random samples for different models with truncated Normal approximation against the exact PDF and CDF of the model} \label{fig:TN_cdf_pdf}
\end{figure}

\subsection{Application in Risk Measurement}
The RSBM provides a versatile and powerful tool for modeling various asymmetric behaviors in financial markets. By incorporating skewness, it allows for more accurate pricing, risk management, and strategic decision-making, reflecting the true nature of financial assets and market dynamics. In markets, assets often have returns that are not symmetrically distributed. RSBM can better capture the empirical characteristics of asset returns, providing more realistic pricing and risk management models \cite{bai2019empirical}. In addition, Investors with preferences for or against skewness in returns can benefit from portfolio optimization models incorporating skew Brownian motion. This allows for more tailored portfolio construction, balancing risk and return according to the investor’s skewness preferences. In the risk assessment aspect, conventional risk measures such as Value-at-Risk (VaR) and Conditional Value-at-Risk (CVaR) can be adapted to account for skewness in the underlying return distributions, providing more accurate assessments of downside risk.

Calculating risk measures like VaR and  CVaR based on the theoretical CDF of RSBM is challenging. Therefore, we use the truncated normal approximation (described in the previous section) to evaluate these risk measures based on the estimated parameters. Let $L$ be the loss in a portfolio that follows an RSBM, for a confidence level $q$, $\mathrm{VaR}_q(L)$ defined as:
\[
\mathds{P}(L \geq \mathrm{VaR}_q(L)) = 1 - q.
\]
One can easily calculate the VaR based on estimated parameters of mixed truncated normal approximation with:
\begin{equation} \label{eq:var}
    \mathrm{VaR}_q(L) =
    \begin{cases}
        \mu_1 + \sigma_1 \Phi^{-1} \left( \frac{q}{\alpha} \Phi(\frac{-\mu_1}{\sigma_1}) \right), & q \leq \alpha, \\
        \mu_2 + \sigma_2 \Phi^{-1}\left( \frac{q-\alpha}{1-\alpha}  + \frac{1-q}{1-\alpha} \Phi(\frac{-\mu_2}{\sigma_2})  \right),  & q > \alpha,
    \end{cases}
\end{equation}

Since the VaR of loss values is a positive value, the CVaR can be calculated as follows:
\begin{equation} \label{eq:cvar}
    \begin{split}
        \mathrm{CVaR}_q(L) = \mathds{E}[L|L\geq \mathrm{VaR}_q(L)] &= \frac{1}{1-q} \int_{\mathrm{VaR}_q(L)}^{\infty} x h(x) \mathrm{d} x  \\
        &= \frac{1}{1-q} \int_{\mathrm{VaR}_q(L)}^{\infty} \frac{1-\alpha}{1-G(0)} x g(x) \mathrm{d} x \\
        &= \frac{1-\alpha}{(1-\Phi(\frac{-\mu_2}{\sigma_2}))(1-q)} \int_{\mathrm{VaR}_q(L)}^{\infty} x g(x) \mathrm{d} x \\
        &= \mu_2 + \frac{\sigma_2(1-\alpha)}{(1-\Phi(\frac{-\mu_2}{\sigma_2}))(1-q)} \phi \left( \Phi^{-1}\left( \frac{q-\alpha}{1-\alpha}  + \frac{1-q}{1-\alpha} \Phi(\frac{-\mu_2}{\sigma_2}) \right) \right).
    \end{split}
\end{equation}
We utilize Model 2 in Table (\ref{tab:sbm_models}) to evaluate the accuracy of the approximated VaR and CVaR. First, we compute the CDF of the RSBM using 2500 data points from the interval $[-20, 20]$. We next calculate the VaR using the linear interpolation method on the CDF and data points. We then use the estimated parameters of the mixed truncated distribution to approximate the VaR and CVaR. Finally, we run 100,000 quasi-random samples through the historical Monte Carlo simulation to calculate the VaR and CVaR for different confidence levels greater than $\alpha = 0.2569$ which acts like a border to have positive VaR. Note that we do not report negative VaR and CVaR because negative loss is not considered a risk.

Figure (\ref{fig:var_cvar}) presents the graph of estimated VaR and CVaR using the formulas in (\ref{eq:var}) and (\ref{eq:cvar}), as well as the historical Monte Carlo simulation. This plot compares the precision of the approximated value of risk to the semi-exact values, which are displayed in black. The mean square of errors of approximated VaR and Monte-Carlo VaR are 0.000019 and 0.000025 respectively. 
\begin{figure}
    \centering
    \includegraphics[width=0.5\linewidth]{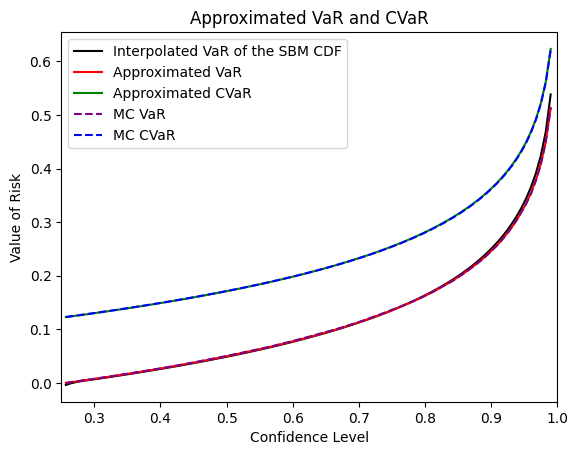}
    \caption{Evaluating approximated VaR and CVaR using mixed truncated normal distribution and Monte-Carlo (MC) simulation methods for values greater than $\alpha = 0.2569$.}
    \label{fig:var_cvar}
\end{figure}

\newpage
%\begin{proposition}
%    If $X_t$ follows a skew Brownian motion dynamic, in which the start point and skew level are $0$, the jump size of discontinuity in $0$ is given by:
%    \begin{equation*}
 %       4 \beta \int_{0}^{\infty} \int_{0}^{t} h(t- \tau; (1+\beta) b , \mu_+) h(\tau; (1-\beta)b, -\mu_-) \mathrm{d}\tau \mathrm{d}b
%    \end{equation*}
%\end{proposition}
%{\color{blue}Move it to Section IV. Discuss what affects the jump size. What happens if $t\to\infty$?}    

%{\color{blue} The references are not alphabetically listed!}

\bibliographystyle{dinat}
%argument is your BibTeX string definitions and bibliography database(s)
\bibliography{references}

\begin{thebibliography}{35}
% this bibliography was produced with the style dinat.bst v2.5
\makeatletter
\newcommand{\dinatlabel}[1]%
{\ifNAT@numbers\else\NAT@biblabelnum{#1}\hspace{2\labelsep}\fi}
\makeatother
\expandafter\ifx\csname natexlab\endcsname\relax\def\natexlab#1{#1}\fi
\expandafter\ifx\csname url\endcsname\relax\def\url#1{\texttt{#1}}\fi

\bibitem[Alexander und Poularikas(1998)]{alexander1998handbook}
\dinatlabel{Alexander und Poularikas 1998} \textsc{Alexander}, Ed~; \textsc{Poularikas}, D:
\newblock The handbook of formulas and tables for signal processing.
\newblock In: \emph{Boca Raton, Florida}
\newblock 33431 (1998)

\bibitem[Appuhamillage u.\,a.(2011)Appuhamillage, Bokil, Thomann, Waymire und Wood]{appuhamillage2011occupation}
\dinatlabel{Appuhamillage u.\,a. 2011} \textsc{Appuhamillage}, Thilanka~; \textsc{Bokil}, Vrushali~; \textsc{Thomann}, Enrique~; \textsc{Waymire}, Edward~; \textsc{Wood}, Brian:
\newblock Occupation and local times for skew Brownian motion with applications to dispersion across an interface.
\newblock In: \emph{Annals of Applied Probability}
\newblock (2011)

\bibitem[Ascione u.\,a.(2024)Ascione, Bufalo und Orlando]{ascione2024modeling}
\dinatlabel{Ascione u.\,a. 2024} \textsc{Ascione}, Giacomo~; \textsc{Bufalo}, Michele~; \textsc{Orlando}, Giuseppe:
\newblock Modeling volatility of disaster-affected populations: A non-homogeneous geometric-skew Brownian motion approach.
\newblock In: \emph{Communications in Nonlinear Science and Numerical Simulation}
\newblock 130 (2024), S.~107761

\bibitem[Bai und Guo(2019)]{bai2019empirical}
\dinatlabel{Bai und Guo 2019} \textsc{Bai}, Yizhou~; \textsc{Guo}, Zhiyu:
\newblock An empirical investigation to the “skew” phenomenon in stock index markets: evidence from the Nikkei 225 and others.
\newblock In: \emph{Sustainability}
\newblock 11 (2019), Nr.~24, S.~7219

\bibitem[Bai u.\,a.(2022)Bai, Wang, Zhang und Zhuo]{bai2022bayesian}
\dinatlabel{Bai u.\,a. 2022} \textsc{Bai}, Yizhou~; \textsc{Wang}, Yongjin~; \textsc{Zhang}, Haoyan~; \textsc{Zhuo}, Xiaoyang:
\newblock Bayesian estimation of the skew Ornstein-Uhlenbeck process.
\newblock In: \emph{Computational Economics}
\newblock (2022), S.~1--49

\bibitem[Bardou und Martinez(2010)]{bardou2010statistical}
\dinatlabel{Bardou und Martinez 2010} \textsc{Bardou}, Olivier~; \textsc{Martinez}, Miguel:
\newblock Statistical estimation for reflected skew processes.
\newblock In: \emph{Statistical inference for stochastic processes}
\newblock 13 (2010), S.~231--248

\bibitem[Berezin und Zayats(2017)]{berezin2017skew}
\dinatlabel{Berezin und Zayats 2017} \textsc{Berezin}, Sergey~; \textsc{Zayats}, Oleg:
\newblock Skew Brownian motion with dry friction: The Pugachev-Sveshnikov approach.
\newblock In: \emph{arXiv preprint arXiv:1705.10980}
\newblock (2017)

\bibitem[Borodin(2023)]{borodin2023distributions}
\dinatlabel{Borodin 2023} \textsc{Borodin}, AN:
\newblock Distributions of Functionals of a Skew Brownian motion with Discontinuous Drift.
\newblock In: \emph{Journal of Mathematical Sciences}
\newblock 273 (2023), Nr.~5, S.~676--686

\bibitem[Borodin und Salminen(2015)]{borodin2015handbook}
\dinatlabel{Borodin und Salminen 2015} \textsc{Borodin}, Andrei~N.~; \textsc{Salminen}, Paavo:
\newblock \emph{Handbook of Brownian motion-facts and formulae}.
\newblock Springer Science \& Business Media, 2015

\bibitem[Chen u.\,a.(2024)Chen, Wu und Zhou]{chen2024optimal}
\dinatlabel{Chen u.\,a. 2024} \textsc{Chen}, Zengjing~; \textsc{Wu}, Panyu~; \textsc{Zhou}, Xiaowen:
\newblock Optimal State Equation for the Control of a Diffusion with Two Distinct Dynamics.
\newblock In: \emph{arXiv preprint arXiv:2404.07618}
\newblock (2024)

\bibitem[Corns und Satchell(2007)]{corns2007skew}
\dinatlabel{Corns und Satchell 2007} \textsc{Corns}, TRA~; \textsc{Satchell}, SE:
\newblock Skew Brownian motion and pricing European options.
\newblock In: \emph{The European Journal of Finance}
\newblock 13 (2007), Nr.~6, S.~523--544

\bibitem[Decamps u.\,a.(2006)Decamps, Goovaerts und Schoutens]{decamps2006self}
\dinatlabel{Decamps u.\,a. 2006} \textsc{Decamps}, Marc~; \textsc{Goovaerts}, Marc~; \textsc{Schoutens}, Wim:
\newblock Self exciting threshold interest rates models.
\newblock In: \emph{International Journal of Theoretical and Applied Finance}
\newblock 9 (2006), Nr.~07, S.~1093--1122

\bibitem[Duffie und Singleton(1999)]{duffie1999modeling}
\dinatlabel{Duffie und Singleton 1999} \textsc{Duffie}, Darrell~; \textsc{Singleton}, Kenneth~J.:
\newblock Modeling term structures of defaultable bonds.
\newblock In: \emph{The review of financial studies}
\newblock 12 (1999), Nr.~4, S.~687--720

\bibitem[Freidlin und Sheu(2000)]{freidlin2000diffusion}
\dinatlabel{Freidlin und Sheu 2000} \textsc{Freidlin}, Mark~; \textsc{Sheu}, Shuenn-Jyi:
\newblock Diffusion processes on graphs: stochastic differential equations, large deviation principle.
\newblock In: \emph{Probability theory and related fields}
\newblock 116 (2000), S.~181--220

\bibitem[Gairat und Shcherbakov(2017)]{gairat2017density}
\dinatlabel{Gairat und Shcherbakov 2017} \textsc{Gairat}, Alexander~; \textsc{Shcherbakov}, Vadim:
\newblock Density of skew Brownian motion and its functionals with application in finance.
\newblock In: \emph{Mathematical Finance}
\newblock 27 (2017), Nr.~4, S.~1069--1088

\bibitem[Gairat und Shcherbakov(2022)]{gairat2022skew}
\dinatlabel{Gairat und Shcherbakov 2022} \textsc{Gairat}, Alexander~; \textsc{Shcherbakov}, Vadim:
\newblock Skew Brownian motion with dry friction: Joint density approach.
\newblock In: \emph{Statistics \& Probability Letters}
\newblock 187 (2022), S.~109511

\bibitem[Gao u.\,a.(2024)Gao, Zhou und Lv]{gao2024optimal}
\dinatlabel{Gao u.\,a. 2024} \textsc{Gao}, Zhongqin~; \textsc{Zhou}, Xiaowen~; \textsc{Lv}, Yan:
\newblock De Finetti’s Control for Refracted Skew Brownian Motion.
\newblock In: \emph{arXiv preprint arXiv:2402.11471}
\newblock (2024)

\bibitem[Gibson u.\,a.(2023)Gibson, Hoerger und Kroese]{gibson2023flow}
\dinatlabel{Gibson u.\,a. 2023} \textsc{Gibson}, Lachlan~; \textsc{Hoerger}, Marcus~; \textsc{Kroese}, Dirk:
\newblock A Flow-Based Generative Model for Rare-Event Simulation.
\newblock In: \emph{arXiv preprint arXiv:2305.07863}
\newblock (2023)

\bibitem[Harrison und Shepp(1981)]{harrison1981skew}
\dinatlabel{Harrison und Shepp 1981} \textsc{Harrison}, John~M.~; \textsc{Shepp}, Lawrence~A.:
\newblock On skew Brownian motion.
\newblock In: \emph{The Annals of probability}
\newblock (1981), S.~309--313

\bibitem[Hofert u.\,a.(2021)Hofert, Prasad und Zhu]{hofert2021quasi}
\dinatlabel{Hofert u.\,a. 2021} \textsc{Hofert}, Marius~; \textsc{Prasad}, Avinash~; \textsc{Zhu}, Mu:
\newblock Quasi-random sampling for multivariate distributions via generative neural networks.
\newblock In: \emph{Journal of Computational and Graphical Statistics}
\newblock 30 (2021), Nr.~3, S.~647--670

\bibitem[Horger u.\,a.(2018)Horger, W{\"u}rfl, Christlein und Maier]{horger2018deep}
\dinatlabel{Horger u.\,a. 2018} \textsc{Horger}, Felix~; \textsc{W{\"u}rfl}, Tobias~; \textsc{Christlein}, Vincent~; \textsc{Maier}, A:
\newblock Deep learning for sampling from arbitrary probability distributions.
\newblock In: \emph{arXiv preprint arXiv:1801.04211}
\newblock (2018)

\bibitem[It{\^o} u.\,a.(2012)It{\^o}, Henry~Jr u.\,a.]{ito2012diffusion}
\dinatlabel{It{\^o} u.\,a. 2012} \textsc{It{\^o}}, Kiyosi~; \textsc{Henry~Jr}, P u.\,a.:
\newblock \emph{Diffusion processes and their sample paths}.
\newblock Springer Science \& Business Media, 2012

\bibitem[Jarrow u.\,a.(1997)Jarrow, Lando und Turnbull]{jarrow1997markov}
\dinatlabel{Jarrow u.\,a. 1997} \textsc{Jarrow}, Robert~A.~; \textsc{Lando}, David~; \textsc{Turnbull}, Stuart~M.:
\newblock A Markov model for the term structure of credit risk spreads.
\newblock In: \emph{The review of financial studies}
\newblock 10 (1997), Nr.~2, S.~481--523

\bibitem[Karatzas und Shreve(1984)]{karatzas1984trivariate}
\dinatlabel{Karatzas und Shreve 1984} \textsc{Karatzas}, Ioannis~; \textsc{Shreve}, Steven~E.:
\newblock Trivariate density of Brownian motion, its local and occupation times, with application to stochastic control.
\newblock In: \emph{The Annals of Probability}
\newblock (1984), S.~819--828

\bibitem[Karlin(2014)]{karlin2014first}
\dinatlabel{Karlin 2014} \textsc{Karlin}, Samuel:
\newblock \emph{A first course in stochastic processes}.
\newblock Academic press, 2014

\bibitem[Kraft(1988)]{kraft1988software}
\dinatlabel{Kraft 1988} \textsc{Kraft}, Dieter:
\newblock A software package for sequential quadratic programming.
\newblock In: \emph{Forschungsbericht- Deutsche Forschungs- und Versuchsanstalt fur Luft- und Raumfahrt}
\newblock (1988)

\bibitem[Le~Gall(1984)]{LeGall19893}
\dinatlabel{Le~Gall 1984} \textsc{Le~Gall}, J.-F:
\newblock One-Dimensional Stochastic Equations Involving the Local Times of the Unknown Process.
\newblock In: \emph{Proceedings of the International Conference on Stochastic Analysis and Applications, Swansea, 11-15 April 1983, Lecture Notes in Mathematics}
\newblock 1095 (1984), S.~51--82

\bibitem[Lejay(2006)]{lejay2006constructions}
\dinatlabel{Lejay 2006} \textsc{Lejay}, Antoine:
\newblock On the constructions of the skew Brownian motion.
\newblock
\newblock (2006)

\bibitem[Lejay und Mazzonetto(2023)]{lejay2023maximum}
\dinatlabel{Lejay und Mazzonetto 2023} \textsc{Lejay}, Antoine~; \textsc{Mazzonetto}, Sara:
\newblock Maximum likelihood estimator for skew Brownian motion: the convergence rate.
\newblock In: \emph{Scandinavian Journal of Statistics}
\newblock (2023)

\bibitem[Piessens u.\,a.(2012)Piessens, de~Doncker-Kapenga, {\"U}berhuber und Kahaner]{piessens2012quadpack}
\dinatlabel{Piessens u.\,a. 2012} \textsc{Piessens}, Robert~; \textsc{Doncker-Kapenga}, Elise de~; \textsc{{\"U}berhuber}, Christoph~W.~; \textsc{Kahaner}, David~K.:
\newblock \emph{Quadpack: a subroutine package for automatic integration}. Bd.~1.
\newblock Springer Science \& Business Media, 2012

\bibitem[Ramirez(2011)]{ramirez2011multi}
\dinatlabel{Ramirez 2011} \textsc{Ramirez}, Jorge:
\newblock Multi-skewed Brownian motion and diffusion in layered media.
\newblock In: \emph{Proceedings of the American Mathematical Society}
\newblock 139 (2011), Nr.~10, S.~3739--3752

\bibitem[Walsh(1978)]{walsh1978diffusion}
\dinatlabel{Walsh 1978} \textsc{Walsh}, John~B.:
\newblock A diffusion with a discontinuous local time.
\newblock In: \emph{Ast{\'e}risque}
\newblock 52 (1978), Nr.~53, S.~37--45

\bibitem[Wang u.\,a.(2015)Wang, Song und Wang]{wang2015skew}
\dinatlabel{Wang u.\,a. 2015} \textsc{Wang}, Suxin~; \textsc{Song}, Shiyu~; \textsc{Wang}, Yongjin:
\newblock Skew Ornstein--Uhlenbeck processes and their financial applications.
\newblock In: \emph{Journal of Computational and Applied Mathematics}
\newblock 273 (2015), S.~363--382

\bibitem[Zhang(2000)]{zhang2000calculation}
\dinatlabel{Zhang 2000} \textsc{Zhang}, Ming:
\newblock Calculation of diffusive shock acceleration of charged particles by skew Brownian motion.
\newblock In: \emph{The Astrophysical Journal}
\newblock 541 (2000), Nr.~1, S.~428

\bibitem[Zhu und He(2018)]{zhu2018new}
\dinatlabel{Zhu und He 2018} \textsc{Zhu}, Song-Ping~; \textsc{He}, Xin-Jiang:
\newblock A new closed-form formula for pricing European options under a skew Brownian motion.
\newblock In: \emph{The European Journal of Finance}
\newblock 24 (2018), Nr.~12, S.~1063--1074

\end{thebibliography}

\newpage

\appendix
\section{Proof of Corollary \ref{corr:sbm_one_drift} } \label{app:proof_one_drift}
%The proof of corollary (\ref{corr:sbm_one_drift} 
For  $\mu_- = \mu_+ = \mu$ and $y > 0$,
\begin{eqnarray*}
    & &\mathds{P}(X_t \in \mathrm{d}y) \\
    &=& 2 (1+ \beta)  e^{2 \mu y} \int_{0}^{\infty} \int_{0}^{t} h(t-\tau; (1+\beta)b + y, \mu) h(\tau; (1-\beta)b,-\mu) \mathrm{d}\tau \mathrm{d}b \mathrm{d}y \\
    &=& 2 (1+ \beta)  e^{2 \mu y} \int_{0}^{\infty} \int_{0}^{t} h(t-\tau; (1+\beta)b + y, \mu) h(\tau; (1-\beta)b,\mu) e^{2\mu (1-\beta)b} \mathrm{d}\tau \mathrm{d}b \mathrm{d}y \\
    % &=& 2 (1+ \beta)  e^{2 \mu y} \int_{0}^{\infty} e^{2\mu (1-\beta)b} h(t-\tau; (1+\beta)b + y + (1-\beta)b, \mu)  \mathrm{d}b \mathrm{d}y \\
    &=& 2 (1+ \beta)  e^{2 \mu y} \int_{0}^{\infty} e^{2\mu (1-\beta)b} h(t; y + 2b, \mu)  \mathrm{d}b \mathrm{d}y \\
    &=& \frac{2 (1+ \beta)}{\sqrt{2 \pi t^3}}  e^{2 \mu y} \int_{0}^{\infty} e^{2\mu (1-\beta)b} (y+2b) e^{\frac{-(y+2b+\mu t)^2}{2t}}  \mathrm{d}b \mathrm{d}y \\
    % &=& \frac{2 (1+ \beta)}{\sqrt{2 \pi t^3}}  e^{2 \mu y} \int_{0}^{\infty} (y+2b) e^{\frac{-y^2 - (2b)^2 - \mu^2 t^2 - 2(2b)y - 2y \mu t - 2(2b) \mu t + 4(1-\beta)b \mu t}{2t}}  \mathrm{d}b \mathrm{d}y\\
    % &=& \frac{2 (1+ \beta)}{\sqrt{2 \pi t^3}}  e^{2 \mu y} e^{\frac{-(y+\mu t)^2}{2t}} \int_{0}^{\infty} (y+2b) e^{\frac{- (2b)^2  - 2(2b)y  - 2(2b)  \beta \mu t - (y+\beta \mu t)^2 + (y+\beta \mu t)^2}{2t}} \mathrm{d}b \mathrm{d}y \\
    % &=& \frac{2 (1+ \beta)}{\sqrt{2 \pi t^3}}  e^{2 \mu y} e^{\frac{-(y+\mu t)^2}{2t}} \int_{0}^{\infty} (y+2b) e^{\frac{-(2b+y + \beta \mu t )^2}{2t}} e^{\frac{(y+\beta \mu t)^2}{2t}} \mathrm{d}b \mathrm{d}y \\
    &=& \frac{2 (1+ \beta)}{\sqrt{2 \pi t^3}}  e^{2 \mu y} e^{\frac{-(y+\mu t)^2}{2t}} e^{\frac{(y+\beta \mu t)^2}{2t}}\int_{0}^{\infty} (y+2b) e^{\frac{-(2b+y + \beta \mu t )^2}{2t}}  \mathrm{d}b \mathrm{d}y \\
    &=& \frac{2 (1+ \beta)}{\sqrt{2 \pi t^3}}  e^{\frac{-(y-\mu t)^2}{2t}} e^{\frac{(y+\beta \mu t)^2}{2t}} \Big[ \int_{0}^{\infty} (y+2b + \beta \mu t) e^{\frac{-(2b+y + \beta \mu t )^2}{2t}}  \mathrm{d}b - \beta \mu t \int_{0}^{\infty}  e^{\frac{-(2b+y + \beta \mu t )^2}{2t}}  \mathrm{d}b \Big] \mathrm{d}y \\ 
    &=& \frac{ (1+ \beta)}{\sqrt{2 \pi t^3}}  e^{\frac{-(y-\mu t)^2}{2t}} e^{\frac{(y+\beta \mu t)^2}{2t}} \Big[ \int_{y}^{\infty} (v+ \beta \mu t) e^{\frac{-(v + \beta \mu t )^2}{2t}}  \mathrm{d}v - \beta \mu t \int_{y}^{\infty}  e^{\frac{-(v + \beta \mu t )^2}{2t}}  \mathrm{d}v\Big] \mathrm{d}y \\
    % &=& \frac{ (1+ \beta)}{\sqrt{2 \pi t^3}}  e^{\frac{-(y-\mu t)^2}{2t}} e^{\frac{(y+\beta \mu t)^2}{2t}} \Big[ t e^{\frac{-(y+\beta \mu t)^2}{2t}} - \beta \mu t \sqrt{2t} \frac{\sqrt{\pi}}{2}  \mathrm{erfc}(\frac{y+\beta \mu t}{\sqrt{2t}}) \Big] \mathrm{d}y \\
    &=& \frac{ (1+ \beta)}{\sqrt{2 \pi t}}  e^{\frac{-(y-\mu t)^2}{2t}} e^{\frac{(y+\beta \mu t)^2}{2t}} \Big[e^{\frac{-(y+\beta \mu t)^2}{2t}} - \frac{\beta \mu}{2} \sqrt{2\pi t}  
  \mathrm{erfc}(\frac{y+\beta \mu t}{\sqrt{2t}}) \Big] \mathrm{d}y \\
 &=& \frac{ (1+ \beta)}{\sqrt{2 \pi t}}   e^{\frac{-(y-\mu t)^2}{2t}}  \Big[1 - \frac{\beta \mu}{2} \sqrt{2\pi t} e^{\frac{(y+\beta \mu t)^2}{2t}} 
  \mathrm{erfc}(\frac{y+\beta \mu t}{\sqrt{2t}}) \Big] \mathrm{d}y. 
\end{eqnarray*}

%\newpage
For  $\mu_- = \mu_+ = \mu$ and $y < 0$,
\begin{eqnarray*}
    & &\mathds{P}(X_t \in \mathrm{d}y) \\
    &=& 2 (1- \beta)  e^{2 \mu y} \int_{0}^{\infty} \int_{0}^{t} h(t-\tau; (1+\beta)b , \mu) h(\tau; (1-\beta)b - y,-\mu) \mathrm{d}\tau \mathrm{d}b \mathrm{d}y\\
    &=& 2 (1- \beta)  e^{2 \mu y} \int_{0}^{\infty} \int_{0}^{t} h(t-\tau; (1+\beta)b , -\mu) h(\tau; (1-\beta)b - y,-\mu) e^{-2\mu (1+\beta)b} \mathrm{d}\tau \mathrm{d}b \mathrm{d}y \\
    % &=& 2 (1- \beta)  e^{2 \mu y} \int_{0}^{\infty} e^{-2\mu (1+\beta)b} h(t-\tau; (1+\beta)b + (1-\beta)b-y, -\mu)  \mathrm{d}b \mathrm{d}y\\
    &=& 2 (1- \beta)  e^{2 \mu y} \int_{0}^{\infty} e^{-2\mu (1+\beta)b} h(t; -y + 2b, -\mu)  \mathrm{d}b \mathrm{d}y\\
    &=& \frac{2 (1- \beta)}{\sqrt{2 \pi t^3}}  e^{2 \mu y} \int_{0}^{\infty} e^{-2\mu (1+\beta)b} (-y+2b) e^{\frac{-(-y+2b-\mu t)^2}{2t}}  \mathrm{d}b \mathrm{d}y \\
    % &=& \frac{2 (1- \beta)}{\sqrt{2 \pi t^3}}  e^{2 \mu y} \int_{0}^{\infty} (-y+2b) e^{\frac{-y^2 - (2b)^2 - \mu^2 t^2 + 2(2b)y - 2y \mu t + 2(2b) \mu t - 4(1+\beta)b \mu t}{2t}}  \mathrm{d}b \mathrm{d}y\\
    % &=& \frac{2 (1- \beta)}{\sqrt{2 \pi t^3}}  e^{2 \mu y} e^{\frac{-(y+\mu t)^2}{2t}} \int_{0}^{\infty} (-y+2b) e^{\frac{- (2b)^2  + 2(2b)y  - 2(2b)  \beta \mu t - (-y+\beta \mu t)^2 + (-y+\beta \mu t)^2}{2t}} \mathrm{d}b \mathrm{d}y \\
    % &=& \frac{2 (1- \beta)}{\sqrt{2 \pi t^3}}  e^{2 \mu y} e^{\frac{-(y+\mu t)^2}{2t}} \int_{0}^{\infty} (-y+2b) e^{\frac{-(2b-y + \beta \mu t )^2}{2t}} e^{\frac{(-y+\beta \mu t)^2}{2t}} \mathrm{d}b \mathrm{d}y \\
    &=& \frac{2 (1- \beta)}{\sqrt{2 \pi t^3}}  e^{2 \mu y} e^{\frac{-(y+\mu t)^2}{2t}} e^{\frac{(-y+\beta \mu t)^2}{2t}}\int_{0}^{\infty} (-y+2b) e^{\frac{-(-y +2b+ \beta \mu t )^2}{2t}}  \mathrm{d}b \mathrm{d}y \\
    &=& \frac{2 (1- \beta)}{\sqrt{2 \pi t^3}}  e^{\frac{-(y-\mu t)^2}{2t}} e^{\frac{(-y+\beta \mu t)^2}{2t}} \Big[ \int_{0}^{\infty} (-y+2b + \beta \mu t) e^{\frac{-(-y + 2b + \beta \mu t )^2}{2t}}  \mathrm{d}b - \beta \mu t \int_{0}^{\infty}  e^{\frac{-(-y + 2b  + \beta \mu t )^2}{2t}}  \mathrm{d}b \Big] \mathrm{d}y \\ 
    &=& \frac{ (1- \beta)}{\sqrt{2 \pi t^3}}  e^{\frac{-(y-\mu t)^2}{2t}} e^{\frac{(-y+\beta \mu t)^2}{2t}} \Big[ \int_{-y}^{\infty} (v+ \beta \mu t) e^{\frac{-(v + \beta \mu t )^2}{2t}}  \mathrm{d}v - \beta \mu t \int_{-y}^{\infty}  e^{\frac{-(v + \beta \mu t )^2}{2t}}  \mathrm{d}v\Big] \mathrm{d}y \\
    % &=& \frac{ (1- \beta)}{\sqrt{2 \pi t^3}}  e^{\frac{-(y-\mu t)^2}{2t}} e^{\frac{(-y+\beta \mu t)^2}{2t}} \Big[ t e^{\frac{-(-y+\beta \mu t)^2}{2t}} - \beta \mu t \sqrt{2t} \frac{\sqrt{\pi}}{2}  \mathrm{erfc}(\frac{-y+\beta \mu t}{\sqrt{2t}}) \Big] \mathrm{d}y \\
    &=& \frac{ (1- \beta)}{\sqrt{2 \pi t}}  e^{\frac{-(y-\mu t)^2}{2t}} e^{\frac{(-y+\beta \mu t)^2}{2t}} \Big[e^{\frac{-(-y+\beta \mu t)^2}{2t}} - \frac{\beta \mu}{2} \sqrt{2\pi t}  
  \mathrm{erfc}(\frac{-y+\beta \mu t}{\sqrt{2t}}) \Big] \mathrm{d}y \\
 &=& \frac{ (1- \beta)}{\sqrt{2 \pi t}}   e^{\frac{-(y-\mu t)^2}{2t}}  \Big[1 - \frac{\beta \mu}{2} \sqrt{2\pi t} e^{\frac{(-y+\beta \mu t)^2}{2t}} 
  \mathrm{erfc}(\frac{-y+\beta \mu t}{\sqrt{2t}}) \Big] \mathrm{d}y.
\end{eqnarray*}

%\newpage

\section{Proof of Proposition \ref{prop:sbm_one_drift}} \label{app:proof_pdf_general}
In this part of the appendix, we find the general transition probability in Proposition (\ref{prop:sbm_one_drift}) by inverting the Laplace transform. 

\begin{remark}
    By formula 183 in the table 2.3 in \cite{alexander1998handbook}, and a property of Laplace transform  one can show that 
    \begin{equation} 
        \int_{0}^{\infty} e^{-st} \frac{1}{\sqrt{2 \pi t}} e^{\frac{-(k-\mu t)^2}{2t}-\mu k} \mathrm{d} t = \frac{1}{\sqrt{2s + \mu^2}} e^{-|k|\sqrt{2s+\mu^2}},
    \end{equation}
which is equivalent to formulas 1.0.5 and 1.0.6 in \cite{borodin2015handbook}.\\
In addition, based on formula 181 in the table 2.3 in \cite{alexander1998handbook}, and using  a property of Laplace transform  one can show that for $k \neq 0$,
    \begin{equation} 
        \int_{0}^{\infty} e^{-st} h(t;k,\mu) e^{\mu k} \mathrm{d} t = e^{-|k| \sqrt{2s+\mu^2}},
    \end{equation}
which is  equivalent to formulas 2.0.1 and 2.0.2 in \cite{borodin2015handbook} and equation (\ref{eq:lap_tr_h}).
\end{remark}
Now using the above remark and Corollary \ref{cor:laplace_convolution}, we can calculate the inverse Laplace transform of
$\mathds{P}_x(X_{ e_{q}}\in \mathrm{d}y)$ in Proposition \ref{prop:pmeasure_x}.
\begin{proof}

For $y > 0$ and $x \geq 0$,
\begin{equation*}
    \begin{split}
        p(x,y,t)  &= \mathcal{L}^{-1}\{ \mathds{P}_x(X_{ e_{q}}\in \mathrm{d}y) \}  \\
        &= e^{\mu_+(y-x)} \Big[ \phi(y-x;\mu_+) e^{-\mu_+(y-x)} - \phi(x+y;\mu_+) e^{-\mu_+ (x+y)} \Big] + 2(1+\beta) e^{2\mu_+ y} f^{+}(t;y) * h(t;x,\mu_+)  \\
        &= \phi(y-x;\mu_+) - e^{-2\mu_+ x} \phi(x+y;\mu_+)  + 2(1+\beta) e^{2\mu_+ y}f^{+}(t;y) * h(t;x,\mu_+).
    \end{split}
\end{equation*}
For the last term:
\begin{eqnarray*}
    &&2(1+\beta) e^{2\mu_+ y}\int_{0}^{t} f^{+}(\eta:y) h(t-\eta;x,\mu_+) \mathrm{d} \eta \\
    &=& 2(1+\beta) e^{2\mu_+ y} \int_{0}^{\infty} \int_{0}^{t}  \int_{0}^{\eta} h(\eta - \tau; (1+\beta)b + y, \mu_+) h(\tau; (1-\beta)b, -\mu_-) h(t-\eta;x,\mu_+) \mathrm{d} \tau   \mathrm{d} \eta \mathrm{d} b \\
    &=& 2(1+\beta) e^{2\mu_+ y} \int_{0}^{\infty} \int_{0}^{t} h(\tau; (1-\beta)b, -\mu_-) \int_{\tau}^{t} h(\eta - \tau; (1+\beta)b + y, \mu_+)  h(t-\eta;x,\mu_+) \mathrm{d} \eta \mathrm{d} \tau  \mathrm{d} b \\
    &=& 2(1+\beta) e^{2\mu_+ y} \int_{0}^{\infty} \int_{0}^{t} h(\tau; (1-\beta)b, -\mu_-) \int_{0}^{t-\tau} h(t-\tau -u; (1+\beta)b + y, \mu_+)  h(u;x,\mu_+) \mathrm{d} u \mathrm{d} \tau  \mathrm{d} b \\
    &=& 2(1+\beta) e^{2\mu_+ y} \int_{0}^{\infty} \int_{0}^{t} h(t-\tau; (1+\beta)b + x + y, \mu_+) h(\tau; (1-\beta)b, -\mu_-)  \mathrm{d} \tau  \mathrm{d} b,
\end{eqnarray*}
 where we change the order of $\tau$ and $\eta$ in the second equality and use changing of variables $u=t- \eta$, in the third equality. 

So,
\begin{equation*}
    p(x,y,t) = \phi(y-x;\mu_+) - e^{-2\mu_+ x} \phi(x+y;\mu_+) + 2(1+\beta) e^{2\mu_+ y} \int_{0}^{\infty} \int_{0}^{t} h(t-\tau; (1+\beta)b + x + y, \mu_+) h(\tau; (1-\beta)b, -\mu_-)  \mathrm{d} \tau  \mathrm{d} b. 
\end{equation*}

For $y > 0$ and $x\leq 0$,
\begin{equation*}
    \begin{split}
        p(x,y,t)  &= \mathcal{L}^{-1}\{ \mathds{P}_x(X_{ e_{q}}\in \mathrm{d}y) \} \\
        &= 2(1+\beta) e^{2\mu_+ y} f^{+}(t;y) * h(t;-x,\mu_-) e^{-2\mu_- x}  \\
        &= 2(1+\beta) e^{2\mu_+ y} \int_{0}^{t} f^{+}(\eta;y) h(t-\eta;-x,-\mu_-) \mathrm{d} \eta \\
        &= 2(1+\beta) e^{2\mu_+ y} \int_{0}^{\infty} \int_{0}^{t}  \int_{0}^{\eta} h(\tau; (1+\beta)b+y, \mu_+) h(\eta-\tau; (1-\beta)b, -\mu_-) h(t-\eta;-x,-\mu_-) \mathrm{d} \tau \mathrm{d} \eta \mathrm{d} b \\
        &= 2(1+\beta) e^{2\mu_+ y} \int_{0}^{\infty} \int_{0}^{t} h(\tau; (1+\beta)b+y, \mu_+)  \int_{\tau}^{t} h(\eta-\tau; (1-\beta)b, -\mu_-) h(t-\eta;-x,-\mu_-) \mathrm{d} \eta \mathrm{d} \tau \mathrm{d} b \\
        &= 2(1+\beta) e^{2\mu_+ y} \int_{0}^{\infty} \int_{0}^{t} h(\tau; (1+\beta)b+y, \mu_+)  \int_{0}^{t-\tau} h(t-\tau -u; (1-\beta)b, -\mu_-) h(u;-x,-\mu_-) \mathrm{d} u \mathrm{d} \tau \mathrm{d} b  \\
        &= 2(1+\beta) e^{2\mu_+ y}  \int_{0}^{\infty} \int_{0}^{t} h(\tau; (1+\beta)b+y, \mu_+)  h(t-\tau; (1-\beta)b -x, -\mu_-) \mathrm{d} \tau \mathrm{d} b,
    \end{split}
\end{equation*}   
 where we change the order of $\tau$ and $\eta$ in the fifth equality and use the change of variable $u=t- \eta$ in the sixth equality.
 
For $y < 0$ and $x \geq 0$,
\begin{equation*}
    \begin{split}
        p(x,y,t)  &= \mathcal{L}^{-1}\{ \mathds{P}_x(X_{ e_{q}}\in \mathrm{d}y) \} \\
        &= 2(1-\beta) e^{2\mu_- y} f^{-}(t;y) * h(t;x,\mu_+)  \\
        &= 2(1-\beta) e^{2\mu_- y} \int_{0}^{t} f^{-}(\eta;y) h(t-\eta;x,\mu_+) \mathrm{d} \eta \\
        &= 2(1-\beta) e^{2\mu_- y} \int_{0}^{\infty} \int_{0}^{t}  \int_{0}^{\eta} h(\eta - \tau; (1+\beta)b, \mu_+) h(\tau; (1-\beta)b-y, -\mu_-) h(t-\eta;x,\mu_+) \mathrm{d} \tau \mathrm{d} \eta \mathrm{d} b \\
        &= 2(1-\beta) e^{2\mu_- y} \int_{0}^{\infty} \int_{0}^{t} h(\tau; (1-\beta)b-y, -\mu_-)  \int_{\tau}^{t} h(\eta - \tau; (1+\beta)b, \mu_+)  h(t-\eta;x,\mu_+) \mathrm{d} \eta \mathrm{d} \tau \mathrm{d} b \\
        &= 2(1-\beta) e^{2\mu_- y} \int_{0}^{\infty} \int_{0}^{t} h(\tau; (1-\beta)b-y, -\mu_-)  \int_{0}^{t-\tau} h(t - \tau -u; (1+\beta)b, \mu_+)  h(u;x,\mu_+) \mathrm{d} u \mathrm{d} \tau \mathrm{d} b \\
        &= 2(1-\beta) e^{2\mu_- y} \int_{0}^{\infty} \int_{0}^{t} h(t - \tau; (1+\beta)b+x, \mu_+) h(\tau; (1-\beta)b-y, -\mu_-)   \mathrm{d} \tau \mathrm{d} b,
    \end{split}
\end{equation*}   
where we change the order of integrals in the fifth equality and use changing the variable $u=t- \eta$ in the sixth equality.

For $y < 0$ and $x \leq 0$,
\begin{equation*}
    \begin{split}
        p(x,y,t)  &= \mathcal{L}^{-1}\{ \mathds{P}_x(X_{ e_{q}}\in \mathrm{d}y) \}  \\
        &= e^{\mu_-(y-x)} \Big[ \phi(y-x;\mu_-) e^{-\mu_-(y-x)} - \phi(-x-y;\mu_-) e^{\mu_- (x+y)} \Big] + 2(1-\beta) e^{2 \mu_- y} f^{-}(t;y) * h(t;-x,\mu_-) e^{-2\mu_- x} \\
        &= \phi(y-x;\mu_-) - e^{2\mu_- y} \phi(-x-y;\mu_-)  + 2(1-\beta) e^{2 \mu_- y} f^{-}(t;y) * h(t;-x,-\mu_-).
    \end{split}
\end{equation*}
For the last term:
\begin{eqnarray*}
    && 2(1-\beta) e^{2 \mu_- y} \int_{0}^{t} f^{-}(\eta;y) h(t-\eta;-x,-\mu_-) \mathrm{d} \eta \\
    &=& 2(1-\beta) e^{2 \mu_- y} \int_{0}^{\infty} \int_{0}^{t}  \int_{0}^{\eta} h(\tau; (1+\beta)b, \mu_+) h(\eta - \tau; (1-\beta)b-y, -\mu_-) h(t-\eta;-x,-\mu_-)  \mathrm{d} \tau   \mathrm{d} \eta \mathrm{d} b \\
    &=& 2(1-\beta) e^{2 \mu_- y} \int_{0}^{\infty} \int_{0}^{t}  h(\tau; (1+\beta)b, \mu_+)  \int_{\tau}^{t} h(\eta - \tau; (1-\beta)b-y, -\mu_-) h(t-\eta;-x,-\mu_-)  \mathrm{d} \eta   \mathrm{d} \tau \mathrm{d} b \\
    &=& 2(1-\beta) e^{2 \mu_- y} \int_{0}^{\infty} \int_{0}^{t}  h(\tau; (1+\beta)b, \mu_+)  \int_{0}^{t-\tau} h(t - \tau- u; (1-\beta)b-y, -\mu_-) h(u;-x,-\mu_-)  \mathrm{d} u   \mathrm{d} \tau \mathrm{d} b \\
    &=& 2(1-\beta) e^{2 \mu_- y} \int_{0}^{\infty} \int_{0}^{t}  h(\tau; (1+\beta)b, \mu_+)  h(t - \tau; (1-\beta)b-y-x, -\mu_-)   \mathrm{d} \tau \mathrm{d} b,
\end{eqnarray*}
 where we change the order of $\tau$ and $\eta$ in the second equality and use changing variable $u=t- \eta$ in the third equality. So, 
\begin{equation*}
\begin{split}
    p(x,y,t) =& \phi(y-x;\mu_-) - e^{2\mu_- y} \phi(-x-y;\mu_-)\\
    &+ 2(1-\beta) e^{2 \mu_- y} \int_{0}^{\infty} \int_{0}^{t}  h(\tau; (1+\beta)b, \mu_+)  h(t - \tau; (1-\beta)b-y-x, -\mu_-)   \mathrm{d} \tau \mathrm{d} b.
    \end{split}
\end{equation*}

\end{proof}

\section{Remarks on the joint density in Gairat and Shcherbakov \cite{gairat2017density} } \label{app:girat}
\textbf{Based  on equation (2.3) in their paper}

In \cite{gairat2017density} the following joint density of the refracted skew Brownian motion $X_T$, its last visit to the origin $\tau$, occupation time $V$, and local times $L_T^{(0)}(X)$ is given in (2.3), where $T$ is time.

For $0 \leq v \leq t \leq T , l\geq 0$,
\begin{equation*}
   \phi_T(t,v,x,l) = \begin{cases}
       2 p h^\prime(v,lp) h^\prime(t-v,lq) h^\prime(T-t,x) e^{-\frac{m_1^2 v}{2} - 
    \frac{m_2^2(T-v)}{2} - l(m_1p - qm_2) + m_1x},  & x>0, \\
    2 q h^\prime(v,lp) h^\prime(t-v,lq) h^\prime(T-t,x) e^{-\frac{m_1^2 v}{2} - 
    \frac{m_2^2(T-v)}{2} - l(m_1p - qm_2) + m_2x},  & x<0,
   \end{cases} 
\end{equation*}
 where $h^\prime(s;y) = \frac{|y|}{\sqrt{2\pi s^3}}e^{\frac{-y^2}{2s}}$, $p = 1-q$ and $m_1$ and $m_2$ correspond to $\mu_-$ and $\mu_+$ respectively in our notation. We can write function $h$ in (\ref{func:h}) as $h(s;y,\mu) = h^{\prime} (s;y) e^{-\frac{\mu^2 s}{2}} e^{-\mu y}$.

Now, let's find the marginal distribution of $X_T$, for $x>0$:
\begin{equation*}
    \begin{split}
        \mathds{P}_T(x) &= \int_0^{\infty} \int_{0}^{T} \int_{0}^{t} \phi_T(t,v,x,l)  \mathrm{d}v \mathrm{d} t \mathrm{d} l \\
        &= \int_0^{\infty} \int_{0}^{T} \int_{0}^{t} 2 p h^\prime(v,lp) h^\prime(t-v,lq) h^\prime(T-t,x) e^{-\frac{m_1^2 v}{2} - \frac{m_2^2(T-v)}{2} - l(m_1p - qm_2) + m_1x}  \mathrm{d}v \mathrm{d} t \mathrm{d} l \\
        &= 2 p e^{m_1x}\int_0^{\infty} \int_{0}^{T} \int_{0}^{t}  h^\prime(v,lp) e^{-\frac{m_1^2 v}{2}} e^{- m_1 lp} h^\prime(t-v,lq) h^\prime(T-t,x) e^{- 
        \frac{m_2^2(T-t+t-v)}{2} + m_2 lq}  \mathrm{d}v \mathrm{d} t \mathrm{d} l \\
        &= 2 p e^{m_1x}\int_0^{\infty} \int_{0}^{T} \int_{0}^{t}  h^\prime(v,lp) e^{-\frac{m_1^2 v}{2}} e^{- m_1 lp} h^\prime(t-v,lq) e^{- 
        \frac{m_2^2(t-v)}{2}} e^{m_2 lq} h^\prime(T-t,x) e^{- 
        \frac{m_2^2(T-t)}{2}}  \mathrm{d}v \mathrm{d} t \mathrm{d} l \\
        &= 2 p e^{m_1x}\int_0^{\infty} \int_{0}^{T} \int_{0}^{t}  h(v;lp,m_1) h(t-v;lq, m_2) e^{2m_2 lq} h(T-t;x,m_2) e^{m_2 x}  \mathrm{d}v \mathrm{d} t \mathrm{d} l .
    \end{split}
\end{equation*}
Now we change the order of $v$ and $t$ in the triple integration:
\begin{equation*}
    \begin{split}
        \mathds{P}_T(x) &= 2 p e^{m_1x}\int_0^{\infty} \int_{0}^{T} \int_{v}^{T}  h(v;lp,m_1) h(t-v;lq, m_2) e^{2m_2 lq} h(T-t;x,m_2) e^{m_2 x}  \mathrm{d}t \mathrm{d} v \mathrm{d} l \\
        &= 2 p e^{m_1x}\int_0^{\infty} \int_{0}^{T} h(v;lp,m_1) e^{2m_2 lq} e^{m_2 x} \int_{v}^{T}  h(t-v;lq, m_2) h(T-t;x,m_2) \mathrm{d}t \mathrm{d} v \mathrm{d} l.
    \end{split}
\end{equation*}
Changing variable $u=T-t$ we have
\begin{equation*}
    \begin{split}
        \mathds{P}_T(x)  &= 2 p e^{m_1x}\int_0^{\infty} \int_{0}^{T} h(v;lp,m_1) e^{2m_2 lq} e^{m_2 x} \int_{0}^{T-v}  h(T-v-u;lq, m_2) h(u;x,m_2) \mathrm{d}u \mathrm{d} v \mathrm{d} l \\
        &= 2 p e^{m_1x}\int_0^{\infty} \int_{0}^{T} h(v;lp,m_1) e^{2m_2 lq} e^{m_2 x} h(T-v;lq+x,m_2)  \mathrm{d} v \mathrm{d} l \\
        &= 2 p e^{m_1x}\int_0^{\infty} \int_{0}^{T} h(v;lp,m_1) h(T-v;lq+x,-m_2) e^{-m_2x}  \mathrm{d} v \mathrm{d} l \\
        &= 2 p e^{(m_1-m_2)x}\int_0^{\infty} \int_{0}^{T} h(v;lp,m_1) h(T-v;lq+x,-m_2)  \mathrm{d} v \mathrm{d} l,
    \end{split}
\end{equation*}
which is not equal to probability density in Theorem (\ref{thm:pdf_sbm}).

for $x<0$:
\begin{equation*}
    \begin{split}
        \mathds{P}_T(x) &= \int_0^{\infty} \int_{0}^{T} \int_{v}^{T} \phi_T(t,v,x,l)  \mathrm{d}t \mathrm{d} v \mathrm{d} l \\
        &= \int_0^{\infty} \int_{0}^{T} \int_{v}^{T} 2 q h^\prime(v,lp) h^\prime(t-v,lq) h^\prime(T-t,x) e^{-\frac{m_1^2 v}{2} - \frac{m_2^2(T-v)}{2} - l(m_1p - qm_2) + m_2 x}  \mathrm{d} t \mathrm{d} v \mathrm{d} l \\
        &= 2 q \int_0^{\infty} \int_{0}^{T} \int_{v}^{T}  h^\prime(v,lp) e^{-\frac{m_1^2 v}{2}} e^{- m_1 lp} h^\prime(t-v,lq) h^\prime(T-t,x) e^{- 
        \frac{m_2^2(T-t+t-v)}{2} + m_2 lq + m_2 x}  \mathrm{d} t \mathrm{d} v \mathrm{d} l \\
        &= 2 q \int_0^{\infty} \int_{0}^{T} \int_{v}^{T}  h^\prime(v,lp) e^{-\frac{m_1^2 v}{2}} e^{- m_1 lp} h^\prime(t-v,lq) e^{- 
        \frac{m_2^2(t-v)}{2}} e^{m_2 lq} h^\prime(T-t,x) e^{- 
        \frac{m_2^2(T-t)}{2}} e^{m_2 x}  \mathrm{d} t \mathrm{d} v \mathrm{d} l \\
        &= 2 q \int_0^{\infty} \int_{0}^{T} \int_{v}^{T}  h(v;lp, m_1) h(t-v;lq , m_2)  e^{2m_2 lq} h(T-t;x,m_2)  e^{2m_2 x}  \mathrm{d} t \mathrm{d} v \mathrm{d} l \\
        &= 2 q \int_0^{\infty} \int_{0}^{T} h(v;lp, m_1) e^{2m_2 lq}  \int_{v}^{T}   h(t-v;lq , m_2)   h(T-t;-x,m_2)    \mathrm{d} t \mathrm{d} v \mathrm{d} l .
    \end{split}
\end{equation*}
By Changing of variable $u = T -t$:
\begin{equation*}
    \begin{split}
       \mathds{P}_T(x) &= 2 q \int_0^{\infty} \int_{0}^{T} h(v;lp, m_1) e^{2m_2 lq}  \int_{0}^{T-v}  h(T-v-u;lq , m_2)   h(u;-x,m_2)    \mathrm{d} u \mathrm{d} v \mathrm{d} l  \\
       &= 2 q \int_0^{\infty} \int_{0}^{T} h(v;lp, m_1) e^{2m_2 lq}  h(T-v;lq -x, m_2)  \mathrm{d} v \mathrm{d} l \\
       &= 2 q \int_0^{\infty} \int_{0}^{T} h(v;lp, m_1) e^{2 m_2 (lq-x)} e^{2 m_2 x}  h(T-v;lq -x, m_2)  \mathrm{d} v \mathrm{d} l \\
       &= 2 q e^{2 m_2 x} \int_0^{\infty} \int_{0}^{T} h(v;lp, m_1)  h(T-v;lq -x, -m_2)  \mathrm{d} v \mathrm{d} l,
    \end{split}
\end{equation*}
which is equal to probability density in Theorem \ref{thm:pdf_sbm}.

\textbf{Based  on equation (2.4) in their paper}

Based on the same paper, the joint density of the refracted skew Brownian motion ($X_T$), its last visit to the origin ($\tau$), total occupation time ($U$), and local times ($L_T^{(0)}(X)$) is given by equation (2.4):
\begin{equation*}
   \varphi_T(t,v,x,l) = \begin{cases}
       2 p h^\prime(u+t-T,lp) h^\prime(T-u,lq) h^\prime(T-t,x) e^{-\frac{m_1^2 u}{2} - 
    \frac{m_2^2(T-u)}{2} - l(m_1p - qm_2) + m_1x},  & x>0, \\
    l>0, t\leq T, T-t \leq u \leq T, & \\
    2 q h^\prime(u,lp) h^\prime(t-u,lq) h^\prime(T-t,x) e^{-\frac{m_1^2 u}{2} - 
    \frac{m_2^2(T-u)}{2} - l(m_1p - qm_2) + m_2x},  & x<0,\\
    l>0, 0 \leq u \leq t \leq T.
   \end{cases} 
\end{equation*}
Now, let's find the marginal distribution of $X_T$, for $x>0$:
\begin{equation*}
    \begin{split}
        \mathds{P}_T(x) &= \int_0^{\infty} \int_{0}^{T} \int_{T-t}^{T} \varphi_T(t,v,x,l)  \mathrm{d}u \mathrm{d} t \mathrm{d} l \\
        &= \int_0^{\infty} \int_{0}^{T} \int_{T-t}^{T} 2 p h^\prime(u+t-T,lp) h^\prime(T-u,lq) h^\prime(T-t,x) e^{-\frac{m_1^2 u}{2} -  \frac{m_2^2(T-u)}{2} - l(m_1p - qm_2) + m_1x} \mathrm{d}u \mathrm{d} t \mathrm{d} l \\
        &= 2 p e^{m_1x} \int_0^{\infty} \int_{0}^{T} \int_{T-t}^{T} h^\prime(u+t-T,lp) h^\prime(T-u,lq) h^\prime(T-t,x) e^{-\frac{m_1^2 u}{2} -  \frac{m_2^2(T-u)}{2} - l(m_1p - qm_2)} \mathrm{d}u \mathrm{d} t \mathrm{d} l \\
        &= 2 p e^{m_1x} \int_0^{\infty} \int_{0}^{T} \int_{T-t}^{T} h^\prime(u+t-T,lp) e^{\frac{-m_1^2}{2}(u+t-T)} e^{-m_1lp} h^\prime(T-u,lq) e^{-\frac{m_2^2(T-u)}{2} + lqm_2} h^\prime(T-t,x) e^{-\frac{m_1^2 (T-t)}{2} } \mathrm{d}u \mathrm{d} t \mathrm{d} l \\
        &= 2 p e^{m_1x} \int_0^{\infty} \int_{0}^{T} \int_{T-t}^{T} h(u+t-T;lp,m_1) h(T-u;lq,m_2) e^{2lqm_2} h^\prime(T-t;x,m_1) e^{m_1x} \mathrm{d}u \mathrm{d} t \mathrm{d} l .
    \end{split}
\end{equation*}
Now we change the order of $u$ and $t$ in the triple integration:
\begin{equation*}
    \begin{split}
        \mathds{P}_T(x) &= 2 p e^{2m_1x} \int_0^{\infty} \int_{0}^{T} \int_{T-u}^{T} h(u+t-T;lp,m_1) h(T-u;lq,m_2) e^{2lqm_2} h^\prime(T-t;x,m_1) \mathrm{d}t \mathrm{d} u \mathrm{d} l .
    \end{split}
\end{equation*}
Using the change of variable $v=T-t$,
\begin{equation*}
    \begin{split}
        \mathds{P}_T(x) &= 2 p e^{2m_1x} \int_0^{\infty} \int_{0}^{T} \int_{0}^{u} h(u-v;lp,m_1) h(T-u;lq,m_2) e^{2lqm_2} h^\prime(v;x,m_1) \mathrm{d}v \mathrm{d} u \mathrm{d} l \\
         &= 2 p e^{2m_1x} \int_0^{\infty} \int_{0}^{T} h(T-u;lq,-m_2)\int_{0}^{u} h(u-v;lp,m_1)  h^\prime(v;x,m_1) \mathrm{d}v \mathrm{d} u \mathrm{d} l \\
         &= 2 p e^{2m_1x} \int_0^{\infty} \int_{0}^{T} h(T-u;lq,-m_2)  h(u;lp+x,m_1)  \mathrm{d} u \mathrm{d} l ,
    \end{split}
\end{equation*}
which is equal to probability density in Theorem (\ref{thm:pdf_sbm}).

Now, let's find the marginal distribution of $X_T$, for $x<0$:
\begin{equation*}
    \begin{split}
        \mathds{P}_T(x) &= \int_0^{\infty} \int_{0}^{T} \int_{u}^{T} \varphi_T(t,v,x,l)  \mathrm{d}u \mathrm{d} t \mathrm{d} l \\
        &= \int_0^{\infty} \int_{0}^{T} \int_{u}^{T} 2 q h^\prime(u,lp) h^\prime(t-u,lq) h^\prime(T-t,x) e^{-\frac{m_1^2 u}{2} - \frac{m_2^2(T-u)}{2} - l(m_1p - qm_2) + m_2x} \mathrm{d}t \mathrm{d} u \mathrm{d} l \\
        &= 2 q \int_0^{\infty} \int_{0}^{T} \int_{u}^{T} h^\prime(u,lp) e^{-\frac{m_1^2 u}{2}-m_1lp} h^\prime(t-u,lq) e^{-\frac{m_2^2 (t-u)}{2}+m_2lq} h^\prime(T-t,x) e^{- \frac{m_2^2(T-t)}{2} + m_2x} \mathrm{d}t \mathrm{d} u \mathrm{d} l \\
        &= 2 q \int_0^{\infty} \int_{0}^{T} \int_{u}^{T} h(u;lp,m_1) h(t-u;lq, m_2) e^{2m_2lq} h(T-t;x,m_2) e^{2m_2x} \mathrm{d}t \mathrm{d} u \mathrm{d} l .
    \end{split}
\end{equation*}
Using the change of variable $v=T-t$,
\begin{equation*}
    \begin{split}
        \mathds{P}_T(x) &= 2 q \int_0^{\infty} \int_{0}^{T} \int_{0}^{T-u} h(u;lp,m_1) h(T-u-v;lq, m_2) e^{2m_2lq} h(v;x,m_2) e^{2m_2x} \mathrm{d}v \mathrm{d} u \mathrm{d} l \\
        &= 2 q \int_0^{\infty} \int_{0}^{T} h(u;lp,m_1)  e^{2m_2lq} \int_{0}^{T-u}  h(T-u-v;lq, m_2) h(v;-x,m_2) \mathrm{d}v \mathrm{d} u \mathrm{d} l \\
        &= 2 q \int_0^{\infty} \int_{0}^{T} h(u;lp,m_1)  e^{2m_2lq}  h(T-u;lq-x, m_2) \mathrm{d} u \mathrm{d} l \\
        &= 2 q e^{2m_2x}\int_0^{\infty} \int_{0}^{T} h(u;lp,m_1)    h(T-u;lq-x, m_2) e^{2m_2lq} e^{-2m_2x}  \mathrm{d} u \mathrm{d} l \\
        &= 2 q e^{2m_2x}\int_0^{\infty} \int_{0}^{T} h(u;lp,m_1)    h(T-u;lq-x, -m_2) \mathrm{d} u \mathrm{d} l,
    \end{split}
\end{equation*}
which is equal to probability density in Theorem (\ref{thm:pdf_sbm}).

\end{document}